\documentclass[12pt]{amsart}
\usepackage[latin1]{inputenc}
\usepackage{color}
\usepackage{bbm}
\usepackage{amsmath,amssymb}
\usepackage{tikz}
\usetikzlibrary{arrows}
\usepackage{enumerate}

\newtheorem{thm}{Theorem}[section]
\newtheorem{cor}[thm]{Corollary}
\newtheorem{lem}[thm]{Lemma}
\newtheorem{prop}[thm]{Proposition}
\theoremstyle{definition}
\newtheorem{defin}[thm]{Definition}

\numberwithin{equation}{section}

\makeatletter
\newcommand{\bR}{\mathbb{R}}

\newcommand{\supp}{\operatorname{supp}}
\newcommand{\dist}{\operatorname{dist}}

\newcommand{\dif}{\,\mathrm{d}}

\newcommand{\charfun}{\ensuremath{\mathbbm 1}}

\DeclareMathOperator{\card}{card}

\allowdisplaybreaks

\begin{document}
\title{Unconditionality of orthogonal spline systems in $H^1$}
\author[G. Gevorkyan]{Gegham Gevorkyan}
\address{Yerevan State University, Alex Manoukian 1, Yerevan, Armenia}
\email{ggg@ysu.am}

\author[A. Kamont]{Anna Kamont}
\address{Institute of Mathematics of the Polish Academy of Sciences, ul. Wita Stwosza 57, 80-952 Gda\'nsk, Poland}
\email{A.Kamont@impan.gda.pl}

\author[K. Keryan]{Karen Keryan}
\address{Yerevan State University, Alex Manoukian 1, Yerevan, Armenia}
\email{karenkeryan@ysu.am}

\author[M. Passenbrunner]{Markus Passenbrunner}
\address{Institute of Analysis, Johannes Kepler University Linz, Austria, 4040 Linz, Altenberger Strasse 69}
\email{markus.passenbrunner@jku.at}

\date{\today}
\begin{abstract}
We give a simple geometric characterization of knot sequences for which the corresponding orthonormal spline system of arbitrary order $k$ is an unconditional basis in the atomic Hardy space $H^1[0,1]$.
\end{abstract}
\maketitle

\section{Introduction.}\label{intro}

This paper belongs to a series of papers studying properties of orthonormal spline systems with arbitrary knots.
The detailed study of such systems started in 1960's with  Z. Ciesielski's papers \cite{Ciesielski1963,Ciesielski1966} on properties of
the Franklin system, which is an orthonormal system consisting of continuous piecewise linear functions with dyadic knots.
Next, the results by J. Domsta (1972), cf. \cite{Domsta1972}, made it possible to extend such study to orthonormal spline systems of higher order
-- and higher smoothness -- with dyadic knots. These systems occurred to be bases or unconditional bases in several function spaces like
$L^p[0,1]$, $1 \leq p < \infty$, $C[0,1]$, $H^p[0,1]$, $0 < p \leq 1$, Sobolev spaces $W^{p,k}[0,1]$, they give characterizations of BMO and VMO spaces, and various  spaces of smooth functions (H\"older functions, Zygmund class, Besov spaces). One should mention here names such as Z. Ciesielski, J. Domsta, S.V. Bochkarev, P. Wojtaszczyk, S.-Y. A. Chang, P. Sj\"olin, J.-O. Str\"omberg (for more detailed references see e.g.
 \cite{GevorkyanKamont1998}, \cite{GevorkyanKamont2005}, \cite{GevorkyanKamont2008}). Nowadays,  results of this kind are known  for wavelets.

The extension of these results
to orthonormal spline systems with arbitrary knots has begun with the case of piecewise linear systems, i.e. general Franklin systems,
 or orthonormal spline systems of order $2$. This was possible due to precise estimates of the inverse to the Gram matrix
 of piecewise linear $B$-spline bases with arbitrary knots, as presented in \cite{KashinSaakyan1989}. First results in this direction were obtained
in \cite{CiesielskiKamont1997} and \cite{GevorkyanKamont1998}.
We would like to mention here two results by G.G. Gevorkyan and A. Kamont.  First, each general Franklin system  is an unconditional
basis in $L^p[0,1]$ for $1 < p < \infty$, cf. \cite{GevKam2004}. Second, there is a simple geometric characterization of knot sequences
for which the corresponding general Franklin system is a basis or an unconditional basis in $H^1[0,1]$, cf. \cite{GevorkyanKamont2005}. We note that in both of these results, an essential tool for their proof is the association of a so called characteristic interval to each general Franklin function $f_n$.

The case of splines of higher order is much more difficult. Let us mention that the basic result -- the existence of a uniform bound
for $L^\infty$-norms of orthogonal projections on spline spaces of order $k$ with arbitrary order (i.e. a bound depending on the order $k$, but independent of
 the sequence of knots)~--~was a long-standing problem known as C. de Boor's conjecture (1973), cf. \cite{deBoor1973}.
 The case of $k=2$ was settled even earlier by Z. Ciesielski \cite{Ciesielski1963}, the cases $k=3,4$ were solved by C. de Boor himself (1968, 1981), cf. \cite{deBoor1968,deBoor1981}, but the
 positive answer in the general case was  given by A. Yu. Shadrin \cite{Shadrin2001} in 2001. A much simplified and shorter proof of this theorem was recently obtained by M.~v.\,Golitschek (2014), cf. \cite{Golitschek2014}.
An immediate consequence of A.Yu. Shadrin's result is that if a sequence of knots is dense in $[0,1]$, then the corresponding
orthonormal spline system of order $k$ is a basis in $L^p[0,1]$, $1 \leq p < \infty$ and $C[0,1]$.
Moreover, Z. Ciesielski \cite{Ciesielski2000} obtained several consequences of Shadrin's result, one of them being some estimate for the inverse to the $B$-spline
Gram matrix. Using this estimate, G.G. Gevorkyan and A. Kamont \cite{GevorkyanKamont2008}  extended  a part of their result from \cite{GevorkyanKamont2005}
to orthonormal spline systems of arbitrary order and obtained a characterization of knot sequences for which the corresponding orthonormal
 spline system of order $k$ is a basis in $H^1[0,1]$. Further extension required more precise estimates for the inverse of $B$-spline Gram matrices,
 of the type known for the piecewise linear case. Such estimates were obtained recently by M. Passenbrunner and A.Yu. Shadrin \cite{PassenbrunnerShadrin2014}.
 Using these estimates, M. Passenbrunner \cite{Passenbrunner2013} proved that for each sequence of knots, the corresponding orthonormal spline system
 of order $k$ is an unconditional basis in $L^p[0,1]$, $1 < p < \infty$.
 The main result of the present paper  is to give a characterization of those knot sequences for which
 the corresponding orthonormal spline system
 of order $k$ is an unconditional basis in $H^1[0,1]$.

\medskip

The paper is organized as follows.
In Section \ref{definitions} we give necessary definitions and we formulate the main result of this paper -- Theorem \ref{thm:uncond}.
In Sections \ref{sec:prel} and \ref{sec:proporth} we recall or prove several facts needed for our results.
In particular, in Section \ref{sec:proporth} we recall precise pointwise estimates for orthonormal
spline systems with arbitrary knots, the associated characteristic intervals and some combinatoric facts for characteristic intervals.
Then Section \ref{four.cond} contains some auxiliary results, and the proof of Theorem \ref{thm:uncond} is done in Section \ref{main.proof}.

\medskip
The results contained in this paper were obtained independently by the two teams G.\,Gevorkyan, K.\,Keryan and A.\,Kamont, M.\,Passenbrunner at the same time, so we have decided to work out a joint paper.
\section{Definitions and the main result.}\label{definitions}

Let $k \geq 2$ be an integer.
In this work, we are concerned with orthonormal spline systems of  order $k$ with arbitrary partitions. We let $\mathcal T =(t_n)_{n=2}^\infty$ be a dense sequence of points in the open unit interval such that each point occurs at most $k$ times. Moreover, define $t_0:=0$ and $t_1:=1$. Such point sequences are called $k$-\emph{admissible}.
For $n$ in the range $-k+2\leq n\leq 1$, let $\mathcal S_n^{(k)}$ be the space of polynomials of order $n+k-1$ (or degree $n+k-2$) on the interval $[0,1]$ and $(f_{n}^{(k)})_{n=-k+2}^{1}$ be the collection of orthonormal polynomials in $L^2\equiv L^2[0,1]$ such that the degree of $f_n^{(k)}$ is $n+k-2$.
For $n\geq 2$, let $\mathcal T_n$ be the ordered sequence of points consisting of the grid points $(t_j)_{j=0}^n$ counting multiplicities and where the knots $0$ and $1$ have multiplicity $k$, i.e., $\mathcal T_n$ is of the form
\begin{align*}
\mathcal T_n=(0=\tau_{n,1}&=\dots=\tau_{n,k}<\tau_{n,k+1}\leq&\\
&\leq\dots\leq\tau_{n,n+k-1}<\tau_{n,n+k}=\dots=\tau_{n,n+2k-1}=1).
\end{align*}
In that case, we also define $\mathcal S_n^{(k)}$ to be the space of polynomial splines of order $k$ with grid points $\mathcal T_n$. For each $n\geq 2$, the space $\mathcal S_{n-1}^{(k)}$ has codimension $1$ in $\mathcal S_{n}^{(k)}$ and, therefore, there exists a function $f_{n}^{(k)}\in \mathcal S_{n}^{(k)}$ that is orthonormal to the space $\mathcal S_{n-1}^{(k)}$. Observe that this function $f_{n}^{(k)}$ is unique up to sign.  
\begin{defin}
The system of functions $(f_{n}^{(k)})_{n=-k+2}^\infty$ is called
\emph{orthonormal spline system of order $k$ corresponding to the sequence
  $(t_n)_{n=0}^\infty$}. 
\end{defin}
  We will frequently omit the parameter $k$ and write $f_n$ and $\mathcal S_n$ instead of $f_{n}^{(k)}$ and $\mathcal S_{n}^{(k)}$, respectively.

Let us note that the case $k=2$ corresponds to orthonormal systems of piecewise linear functions, i.e. general Franklin systems.

We are interested in characterizing sequences of knots $\mathcal T$ such that the system $(f_{n}^{(k)})_{n=-k+2}^\infty$
is an unconditional basis in $H^1=H^1[0,1]$.
By $H^1=H^1[0,1]$ we mean the atomic Hardy space on $[0,1]$, cf \cite{CoifmanWeiss1977}.
 A function $a:[0,1]\to\bR$ is called an \emph{atom}, if either $a\equiv 1$ or there exists an interval $\Gamma$ such that the following conditions are satisfied:
\begin{enumerate}[(i)]
\item $\supp a\subset \Gamma$,
\item $\|a\|_\infty\leq |\Gamma|^{-1}$,
\item $\int_0^1 a(x)\dif x=\int_{\Gamma}a(x)\dif x=0$.
\end{enumerate}
Then, by definition, $H^1$ consists of all functions $f$ that have the representation
\[
f=\sum_{n=1}^\infty c_n a_n
\]
for some atoms $(a_n)_{n=1}^\infty$ and real scalars $(c_n)_{n=1}^\infty$ such that $\sum_{n=1}^\infty |c_n|<\infty$. The space $H^1$ becomes a Banach space under the norm
\[
\|f\|_{H^1}:=\inf \sum_{n=1}^\infty |c_n|,
\]
where $\inf$ is taken over all atomic representations $\sum c_n a_n$ of $f$.

To formulate our result, we need to introduce some regularity conditions for a sequence $\mathcal T$.

For $n\geq 2$, $\ell\leq k$ and $i$ in the range $k-\ell+1\leq i\leq n+k-1$, we define $D_{n,i}^{(\ell)}$ to be the interval $[\tau_{n,i},\tau_{n,i+\ell}]$.
\begin{defin}
Let $\ell\leq k$ and $(t_n)_{n=0}^\infty$ be an $\ell$-admissible (and therefore $k$-admissible) point sequence. Then, this sequence is called \emph{$\ell$-regular with parameter $\gamma\geq 1$} if
\[
\frac{|D_{n,i}^{(\ell)}|}{\gamma}\leq |D_{n,i+1}^{(\ell)}|\leq \gamma|D_{n,i}^{(\ell)}|,\qquad n\geq 2,\  k-\ell+1\leq i \leq n+k-2.
\]
\end{defin}

So, in other words, $(t_n)$ is $\ell$-regular, if there is a uniform finite
bound $\gamma\geq 1$, such that for all $n$, the ratios of the lengths of
neighboring supports of B-spline functions (cf. Section \ref{sec:bsplines}) of order $\ell$ in the grid $\mathcal T_n$ are bounded 
by $\gamma$.

The following characterization for $(f_n^{(k)})$ to be a basis in $H^1$ is the main result of \cite{GevorkyanKamont2008}:
\begin{thm}[\cite{GevorkyanKamont2008}]\label{thm:basis}
Let $k\geq 1$ and let $(t_n)$ be a $k$-admissible sequence of knots in $[0,1]$
with the corresponding orthonormal spline system $(f_n^{(k)})$ of order $k$.
Then, $(f_n^{(k)})$ is a basis in $H^1$ if and only if $(t_n)$ is $k$-regular
with some parameter $\gamma\geq 1$
\end{thm}

In this paper, we prove the characterization for $(f_n^{(k)})$ to be an unconditional basis in $H^1$.
The main result of our paper is the following:

\begin{thm}\label{thm:uncond}
Let $(t_n)$ be a $k$-admissible sequence of points. Then, the corresponding
orthonormal spline system $(f_n^{(k)})$ is an unconditional basis in $H^1$ if
and only if $(t_n)$ satisfies the $(k-1)$-regularity condition with some
parameter $\gamma\geq 1$.
\end{thm}

Let us note that in case $k=2$, i.e. for general Franklin systems, both Theorems \ref{thm:basis} and \ref{thm:uncond} were obtained by G. G. Gevorkyan and A. Kamont in \cite{GevorkyanKamont2005}. (In the terminology of the current paper, the condition of strong regularity from \cite{GevorkyanKamont2005}
is now $1$-regularity, and the condition of strong regularity for pairs from \cite{GevorkyanKamont2005}
is now $2$-regularity.)

The proof of Theorem \ref{thm:uncond} follows the same general scheme as the proof of Theorem 2.2 in \cite{GevorkyanKamont2005}.
In Section \ref{four.cond} we introduce four conditions (A) -- (D) for series with respect to orthonormal spline systems of order $k$
corresponding to a $k$-admissible sequence of points. Then we study relations between these conditions under various regularity assumptions on the underlying sequence of points. Having done this, we proceed with the proof of Theorem  \ref{thm:uncond} in Section \ref{main.proof}.

\section{Preliminaries}\label{sec:prel}
The parameter $k\geq 2$ will always be used for the order of the underlying polynomials or splines.
We use the notation $A(t)\sim B(t)$ to indicate the existence of two constants $c_1,c_2>0$, such that $c_1 B(t)\leq A(t)\leq c_2 B(t)$ for all $t$, where $t$ denotes all implicit and explicit dependencies that the expressions $A$ and $B$ might have. If the constants $c_1,c_2$ depend on an additional parameter $p$, we write this as $A(t)\sim_p B(t)$. Correspondingly, we use the symbols $\lesssim,\gtrsim,\lesssim_p,\gtrsim_p$.  For a subset $E$ of the real line, we denote by $|E|$ the Lebesgue measure of $E$ and by $\charfun_E$ the characteristic function of $E$.
If $f:\Omega\to\bR$ is a real valued function and $\lambda$ is a real parameter,
we write the level set of all points at which $f$ is greater than $\lambda$ as  
$  [f>\lambda] := \{\omega\in\Omega : f(\omega) > \lambda \}. $

\subsection{Properties of regular sequences of points}\label{regularity}

%

The following Lemma describes geometric decay of intervals in regular sequences (recall the notation $D_{n,i}^{(\ell)}=[\tau_{n,i},\tau_{n,i+\ell}]$):

\begin{lem}\label{lem:geom}
Let $(t_n)$ be a $k$-admissible sequence of points that satisfies the $\ell$-regularity condition for some $1\leq \ell\leq k$ with parameter $\gamma$ and let $D_{n_1,i_1}^{(\ell)}\supset\cdots\supset D_{n_{2\ell},i_{2\ell}}^{(\ell)}$ be a strictly decreasing sequence of sets defined above. Then,
\[
|D_{n_{2\ell},i_{2\ell}}^{(\ell)}|\leq \frac{\gamma^\ell}{1+\gamma^\ell} |D_{n_1,i_1}^{(\ell)}|.
\]
\end{lem}
\begin{proof}
We set $V_j:=D_{n_j,i_j}^{(\ell)}$ for $1\leq j\leq 2\ell$. Then, by definition, $V_1$ contains $\ell+1$ grid points from $\mathcal T_{n_1}$ and it contains at least $3\ell$ grid points of the grid $\mathcal T_{n_{2\ell}}$. As a consequence, there exists an interval $D_{n_{2\ell},m}^{(\ell)}$ for some index $m$ that satisfies
\[
\operatorname{int}(D_{n_{2\ell},m}^{(\ell)}\cap V_{2\ell})=\emptyset,\qquad D_{n_{2\ell},m}^{(\ell)}\subset V_{1},\qquad \dist(D_{n_{2\ell},m}^{(\ell)},V_{2\ell})=0.
\]
The $\ell$-regularity of $(t_n)$ now implies
\[
|V_{2\ell}|\leq \gamma^\ell |D_{n_{2\ell},m}^{(\ell)}|\leq\gamma^\ell \big(|V_1|-|V_{2\ell}|\big),
\]
i.e., $|V_{2\ell}|\leq \frac{\gamma^\ell}{1+\gamma^\ell} |V_1|$, which proves the assertion of the lemma.
\end{proof}

\subsection{Properties of {B-spline} functions} \label{sec:bsplines}We define the functions $(N_{n,i}^{(k)})_{i=1}^{n+k-1}$  to be the collection of  B-spline functions of order $k$ corresponding to the partition $\mathcal T_n$. Those functions are normalized in such a way that they form a partition of unity, i.e., $\sum_{i=1}^{n+k-1}N_{n,i}^{(k)}(x)=1$ for all $x\in[0,1]$. Associated to this basis, there exists a biorthogonal basis of $\mathcal S_n$, which is denoted by $(N_{n,i}^{(k)*})_{i=1}^{n+k-1}$. If the setting of the parameters $k$ and $n$ is clear from the context, we also denote those functions by $(N_{i})_{i=1}^{n+k-1}$ and $(N_{i}^*)_{i=1}^{n+k-1}$, respectively.

We will need the following well known formula for the derivative of a linear combination of B-spline functions: if $g=\sum_{j=1}^{n+k-1} a_j N_{n,j}^{(k)}$, then
\begin{equation}\label{eq:splinederivative}
g'=(k-1)\sum_{j=2}^{n+k-1} (a_j-a_{j-1})\frac{N_{n,j}^{(k-1)}}{|D_{n,j}^{(k-1)}|}.
\end{equation}

We now recall an elementary property of polynomials.
\begin{prop}\label{prop:poly}
Let $0<\rho<1$. Let $I$ be an interval and $A\subset I$ be a subset of $I$ with $|A|\geq \rho |I|$. Then, for every polynomial $Q$ of order $k$ on $I$,
\[
\max_{t\in I}|Q(t)|\lesssim_{\rho,k} \sup_{t\in A}|Q(t)|\qquad\text{and}\qquad \int_I |Q(t)|\dif t\lesssim_{\rho,k} \int_A |Q(t)|\dif t.
\]
\end{prop}

We continue with recalling a few important results for B-splines $(N_i)$ and their dual functions $(N_i^*)$.
\begin{prop}\label{prop:lpstab}
Let $1\leq p\leq \infty$ and $g=\sum_{j=1}^{n+k-1} a_j N_j$, where the collection $(N_{i})_{i=1}^{n+k-1}$ are the B-splines of order $k$ corresponding to the partition $\mathcal T_n$. Then,
\begin{equation}\label{eq:lpstab}
|a_j|\lesssim_k |J_j|^{-1/p}\|g\|_{L^p(J_j)},\qquad 1\leq j\leq n+k-1,
\end{equation}
where $J_j$ is a subinterval $[\tau_{n,i},\tau_{n,i+1}]$ of $[\tau_{n,j},\tau_{n,j+k}]$ of maximal length. Additionally,
\begin{equation}\label{eq:deboorlpstab}
\|g\|_p\sim_k \Big(\sum_{j=1}^{n+k-1} |a_j|^p |D_{n,j}^{(k)}|\Big)^{1/p}=\| (a_j|D_{n,j}^{(k)}|^{1/p})_{j=1}^{n+k-1}\|_{\ell^p}.
\end{equation}
Moreover, if $h=\sum_{j=1}^{n+k-1} b_j N_j^*$,
\begin{equation}
\|h\|_p\lesssim_k\Big(\sum_{j=1}^{n+k-1} |b_j|^p |D_{n,j}^{(k)}|^{1-p}\Big)^{1/p}=\|(b_j|D_{n,j}^{(k)}|^{1/p-1})_{j=1}^{n+k-1}\|_{\ell^p}.
\label{eq:lpstabdual}
\end{equation}
\end{prop}

The two inequalites \eqref{eq:lpstab} and \eqref{eq:deboorlpstab} are Lemma 4.1 and Lemma 4.2 in \cite[Chapter 5]{DeVoreLorentz1993}, respectively. Inequality \eqref{eq:lpstabdual} is a consequence of Shadrin's theorem \cite{Shadrin2001}, that the orthogonal projection operator onto $\mathcal S_{n}^{(k)}$ is bounded on $L^\infty$ independently of $n$ and $\mathcal T_n$. For a deduction of \eqref{eq:lpstabdual} from this result, see \cite[Property P.7]{Ciesielski2000}.

The next thing to consider are estimates for the inverse $(b_{ij})_{i,j=1}^{n+k-1}$ of the Gram matrix $(\langle N_{i},N_{j}\rangle)_{i,j=1}^{n+k-1}$. Later, we will need one special property of this matrix, which is that $(b_{ij})_{i,j=1}^{n+k-1}$ is checkerboard, i.e.,
\begin{equation}\label{eq:checkerboard}
(-1)^{i+j} b_{ij}\geq 0\quad \text{for all }i,j.
\end{equation}
This is a simple consequence of the total positivity of the Gram matrix $(\langle N_{i},N_{j}\rangle)_{i,j=1}^{n+k-1}$, cf. \cite{deBoor1968,Karlin1968}.
Moreover, we need the following lower estimate for $b_{i,i}$:
\begin{equation}\label{estbi:lower}
|D^{(k)}_{n,i}|^{-1}  \lesssim_k b_{i,i}.
\end{equation}
 This estimate is a consequence of the total positivity of the $B$-spline Gram matrix, the $L^2$-stability of $B$-splines and the following Lemma \ref{lem:estdiaginverse}

\begin{lem}[\cite{Passenbrunner2013}]\label{lem:estdiaginverse}
Let $C=(c_{ij})_{i,j=1}^n$ be a symmetric positive definite matrix. Then, for $(d_{ij})_{i,j=1}^n=C^{-1}$ we have
\[
c_{ii}^{-1}\leq d_{ii},\qquad 1\leq i\leq n.
\]
\end{lem}

\subsection{Some results for orthonormal spline systems} We recall now two results
concerning  orthonormal spline {series}, which  we will need in the sequel.

\begin{thm}[\cite{PassenbrunnerShadrin2014}]\label{thm:ae}
Let $(f_n)_{n=-k+2}^\infty$ be the orthonormal spline system of order $k$ corresponding to an arbitrary $k$-admissible point sequence $(t_n)_{n=0}^\infty$. Then, for an arbitrary $f\in L^1\equiv L^1[0,1]$, the series $\sum_{n=-k+2}^\infty \langle f,f_n\rangle f_n$ converges to $f$ almost everywhere.
\end{thm}
Let $f\in L^p\equiv L^p[0,1]$ for some $1\leq p<\infty$. Since the orthonormal spline system $(f_n)_{n\geq -k+2}$ is a basis in $L^p$, we can write $f=\sum_{n=-k+2}^\infty a_n f_n$. Based on this expansion, we define the \emph{square function} $Pf:=\big(\sum_{n=-k+2}^\infty |a_n f_n|^2\big)^{1/2}$ and the \emph{maximal function} $Sf:=\sup_m \big| \sum_{n\leq m} a_n f_n \big|$.
Moreover, given a measurable function $g$, we denote by $\mathcal Mg$ the \emph{Hardy-Littlewood maximal function} of $g$ defined as
\[
\mathcal Mg(x):=\sup_{I\ni x} |I|^{-1} \int_I |g(t)|\dif t,
\]
where the supremum is taken over all intervals $I$ containing the point $x$. The connection between the maximal function $Sf$ and the Hardy-Littlewood maximal function is given by the following result:

\begin{thm}[\cite{PassenbrunnerShadrin2014}]\label{thm:maxbound}
If $f\in L^1$, we have
\[
Sf(t)\lesssim_k \mathcal M f(t),\qquad t\in[0,1].
\]
\end{thm}

\section{Properties of orthogonal spline functions and characteristic intervals}\label{sec:proporth}
\subsection{Estimates for $f_n$}
This section treats the calculation and estimation of one explicit orthonormal spline function $f_n^{(k)}$ for fixed $k\in\mathbb N$ and $n\geq 2$ induced by the $k$-admissible sequence $(t_n)_{n=0}^\infty$. {Most of the presented results are taken from \cite{Passenbrunner2013}.}

Here, we change our notation slightly. We fix the parameter $n$ and let $i_0$ be an index with $k+1\leq i_0\leq n+k-1$ such that $\mathcal T_{n-1}$ equals $\mathcal T_n$ with the point $\tau_{i_0}$ removed.
In the points of the partition $\mathcal T_n$, we omit the parameter $n$ and $\mathcal T_n$ is thus given by
\begin{align*}
\mathcal T_n=(0=\tau_1=\dots=\tau_k&<\tau_{k+1}\leq\dots\leq\tau_{i_0}\\
&\leq\dots\leq\tau_{n+k-1}<\tau_{n+k}=\dots=\tau_{n+2k-1}=1).
\end{align*}
We denote by $(N_i:1\leq i\leq n+k-1)$ the B-spline functions corresponding to $\mathcal T_n$.

An (unnormalized) orthogonal spline function $g\in \mathcal S_{n}^{(k)}$ that is orthogonal to $\mathcal S_{n-1}^{(k)}$, as calculated in \cite{Passenbrunner2013}, is given by

\begin{equation}\label{eq:defg}
g=\sum_{j=i_0-k}^{i_0} \alpha_j N_j^* =\sum_{j=i_0-k}^{i_0} \sum_{\ell=1}^{n+k-1} \alpha_j b_{j\ell} N_\ell,
\end{equation}
where $(b_{j\ell})_{j,\ell=1}^{n+k-1}$ is the inverse of the Gram matrix $(\langle N_j,N_\ell\rangle)_{j,\ell=1}^{n+k-1}$ and the sequence $(\alpha_j)$ is given by
\begin{equation}\label{eq:alpha2}
\alpha_j=(-1)^{j-i_0+k}\Big(\prod_{\ell=i_0-k+1}^{j-1}\frac{\tau_{i_0}-\tau_{\ell}}{\tau_{\ell+k}-\tau_{\ell}}\Big)\Big(\prod_{\ell=j+1}^{i_0-1}\frac{\tau_{\ell+k}-\tau_{i_0}}{\tau_{\ell+k}-\tau_{\ell}}\Big),\quad i_0-k\leq j\leq i_0.
\end{equation}
We remark that the sequence $(\alpha_j)$ alternates in sign and since the matrix $(b_{j\ell})_{j,\ell=1}^{n+k-1}$ is checkerboard, we see that the B-spline coefficients of $g$, namely
\begin{equation}\label{eq:defwj}
w_\ell:=\sum_{j=i_0-k}^{i_0} \alpha_j b_{j\ell},\qquad 1\leq \ell\leq n+k-1,
\end{equation}
satisfy
\begin{equation}\label{eq:betragreinziehen}
\Big| \sum_{j=i_0-k}^{i_0}\alpha_j b_{j\ell}\Big|= \sum_{j=i_0-k}^{i_0}|\alpha_j b_{j\ell}|,\qquad 1\leq j\leq n+k-1.
\end{equation}

In the following Definition \ref{def:characteristic}, we assign to each orthonormal spline function a characteristic interval that is a grid point interval $[\tau_i,\tau_{i+1}]$ and lies in the proximity of the newly inserted point $\tau_{i_0}$. The choice of this interval is crucial for proving important properties of the system $(f_n^{(k)})_{n=-k+2}^\infty$. This approach has its origins in \cite{GevKam2004}, where it is proved that general Franklin systems are unconditional bases in $L^p$, $1<p<\infty$.
\begin{defin}\label{def:characteristic}
Let $\mathcal T_{n},\mathcal T_{n-1}$ be as above and $\tau_{i_0}$ be the new point in $\mathcal T_n$ that is not present in $\mathcal T_{n-1}$. We define the \emph{characteristic interval $J_n$ corresponding to the pair $(\mathcal T_n,\mathcal T_{n-1})$} as follows.
\begin{enumerate}
\item
Let
\[
\Lambda^{(0)}:=\{i_0-k\leq j\leq i_0 : |[\tau_j,\tau_{j+k}]|\leq 2\min_{i_0-k\leq \ell\leq i_0}|[\tau_\ell,\tau_{\ell+k}]| \}
\]
be the set of all indices $j$ for which the corresponding support of the B-spline function $N_j$ is approximately minimal. Observe that $\Lambda^{(0)}$ is nonempty.
\item Define
\[
\Lambda^{(1)}:=\{j\in \Lambda^{(0)}: |\alpha_j|=\max_{\ell\in \Lambda^{(0)}} |\alpha_\ell|\}.
\]
For an arbitrary, but fixed index $j^{(0)}\in \Lambda^{(1)}$, set $J^{(0)}:=[\tau_{j^{(0)}},\tau_{j^{(0)}+k}]$.
\item The interval $J^{(0)}$ can now be written as the union of $k$ grid intervals
\[
J^{(0)}=\bigcup_{\ell=0}^{k-1}[\tau_{j^{(0)}+\ell},\tau_{j^{(0)}+\ell+1}]\qquad\text{with }j^{(0)}\text{ as above}.
\]
We define the \emph{characteristic interval} $J_n$ to be one of the above $k$ intervals that has maximal length.
\end{enumerate}
\end{defin}

A few clarifying comments to this definitions are in order. Roughly speaking, we
first take the B-spline support $[\tau_j,\tau_{j+k}]$ intersecting the new point
$\tau_{i_0}$ with minimal length and then we choose as $J_n$ the largest grid
point interval in $[\tau_j,\tau_{j+k}]$. This definition guarantees the
concentration of $f_n$ at $J_n$ in terms of the $L^p$-norm (cf. Lemma
\ref{lem:orthsplineJinterval}) and the exponential decay of $f_n$ away from
$J_n$ (cf. Lemma \ref{lem:lporthspline}), which are
crucial for further investigations. An important ingredient in the proof of Lemma
\ref{lem:orthsplineJinterval} is Proposition \ref{prop:lpstab},
being the reason why we choose the largest grid point interval as $J_n$.
Further important properties of the collection $(J_n)$ of characteristic
intervals are 
that they form a nested family of sets and for a subsequence of decreasing characteristic
intervals, their lengths decay geometrically (cf. Lemma \ref{lem:jinterval}).

Next we remark that the constant $2$ in step (1) of Definition
\ref{def:characteristic} could also be an arbitrary number
$C>1$, but $C=1$ is not allowed. This is in contrast to the definition of
characteristic intervals in \cite{GevKam2004} for piecewise linear orthogonal
functions $(k=2)$, where precisely $C=1$ is chosen,
step~(2) of Definition \ref{def:characteristic} is omitted and $j^{(0)}$ is an
arbitrary index in the set $\Lambda^{(0)}$.

 At first glance, it might seem
 natural to carry over the same definition to arbitrary spline orders $k$, but
 at some point in the proof of Theorem~\ref{thm:estwj}, we estimate
 $\alpha_{j^{(0)}}$ by the constant $(C-1)$ from below, which has to be strictly
 greater than zero in order to establish \eqref{eq:estwj}. Since
 Theorem~\ref{thm:estwj} is also used in the proofs of both
 Lemma~\ref{lem:orthsplineJinterval} and Lemma \ref{lem:lporthspline},
 this is the reason for a different definition of characteristic
 intervals here than in \cite{GevKam2004}, in particular for step~(2) of Definition
 \ref{def:characteristic}. 

\begin{thm}[\cite{Passenbrunner2013}]\label{thm:estwj}
With the above definition \eqref{eq:defwj} of $w_\ell$ for $1\leq \ell\leq n+k-1$ and the index $j^{(0)}$ given in Definition \ref{def:characteristic},
\begin{equation}\label{eq:estwj}
|w_{j^{(0)}}|\gtrsim_k b_{j^{(0)},j^{(0)}}.
\end{equation}
\end{thm}


\begin{lem}[\cite{Passenbrunner2013}]\label{lem:orthsplineJinterval}Let $\mathcal T_n,\,\mathcal T_{n-1}$ be as above and $g$ be the function given in \eqref{eq:defg}. Then, $f_n=g/\|g\|_2$  is the $L^2$-normalized orthogonal spline function corresponding to $(\mathcal T_n,\mathcal T_{n-1})$ and
\[
\|f_n\|_{L^p(J_n)}\sim_k\|f_n\|_p\sim_k |J_n|^{1/p-1/2}\sim_k |J_n|^{1/2}\|g\|_p,\qquad 1\leq p\leq \infty,
\]
where $J_n$ is the characteristic interval associated to $(\mathcal T_n,\mathcal T_{n-1})$.
\end{lem}

By $d_{n} (x)$ we denote the number of points in $\mathcal T_n$ between $x$ and $J_n$ counting endpoints of $J_n$. Correspondingly, for an interval $V\subset [0,1]$, by $d_{n}(V)$ we denote the number of points in $\mathcal T_n$ between $V$ and $J_n$ counting endpoints of both $J_n$ and $V$.

\begin{lem}[\cite{Passenbrunner2013}]\label{lem:lporthspline}
Let $\mathcal T_n,\mathcal T_{n-1}$ be as above, $g=\sum_{j=1}^{n+k-1} w_jN_j$
be the function in \eqref{eq:defg} with $(w_j)_{j=1}^{n+k-1}$ as in
\eqref{eq:defwj} and $f_n = g/\|g\|_2$. Then, there exists a constant $0<q<1$ that depends only on $k$ such that.
\begin{equation}\label{eq:wj}
|w_j|\lesssim_k \frac{q^{d_{n}(\tau_j)}}{|J_n|+\dist(\supp N_j,J_n)+|D_{n,j}^{k}|}\quad\text{for all }1\leq j\leq n+k-1.
\end{equation}
Moreover, if $x<\inf J_n$, we have
\begin{equation}\label{eq:phiplinks}
\|f_n\|_{L^p(0,x)}
\lesssim_k \frac{q^{d_{n}(x)}|J_n|^{1/2}}{(|J_n|+\dist(x,J_n))^{1-1/p}},\qquad 1\leq p\leq \infty.
\end{equation}
Similarly, for $x>\sup J_n$,
\begin{equation}\label{eq:phiprechts}
\|f_n\|_{L^p(x,1)}
\lesssim_k \frac{q^{d_{n}(x)}|J_n|^{1/2}}{(|J_n|+\dist(x,J_n))^{1-1/p}},\qquad 1\leq p\leq \infty.
\end{equation}
\end{lem}

\subsection{Combinatorics of characteristic intervals}\label{sec:comb}
{Next, we recall a combinatorial result about the relative positions of different characteristic intervals:}
\begin{lem}[\cite{Passenbrunner2013}]\label{lem:jinterval}
Let $x,y\in (t_n)_{n=0}^\infty$ such that $x<y$. Then there exists a constant $F_{k}$ only depending on $k$ such that
\[
N_0:=\card\{n:J_n\subseteq [x,y], |J_n|\geq|[x,y]|/2\} \leq F_{k},
\]
where $\card E$ denotes the cardinality of the set $E$.
\end{lem}

{Similarly to \cite{GevKam2004} and \cite{GevorkyanKamont2005}, we need the following estimate involving characteristic intervals and orthonormal spline functions:}
\begin{lem}\label{lem:JnsubsetV}
  Let $(t_n)$ be a $k$-admissible point sequence in $[0,1]$ and let
  $(f_n)_{n\geq -k+2}$ be the corresponding orthonormal spline system of order
  $k$. 
Then, for each interval $V=[\alpha,\beta]\subset [0,1]$,
\[
\sum_{n:J_n\subset V} |J_n|^{1/2} \int_{V^c} |f_n(t)|\dif t\lesssim_k |V|.
\]
\end{lem}
Once we know the  estimates for orthonormal spline functions as in Lemma \ref{lem:lporthspline}
 and the basic combinatorial result for their characteristic intervals, i.e. Lemma \ref{lem:jinterval}, this result follows by the same line of arguments that was used in the proof of Lemma 4.6 in \cite{GevKam2004}, so we skip its proof.

Instead of Lemma 3.4 of \cite{GevorkyanKamont2005}, we will {use} the following:

\begin{lem}\label{lem:sumJn}
Let $(t_n)_{n=0}^\infty$ be a $k$-admissible knot sequence that satisfies the $(k-1)$-regularity condition and let $\Delta=D_{m,i}^{(k-1)}$ for some indices $m$ and $i$. For $\ell\geq 0$, let
\begin{align*}
N(\Delta)&:=\{n:\ \card (\Delta\cap\mathcal T_n)=k,\ J_n\subset\Delta\}, \\
M(\Delta,\ell)&:= \{n: d_n(\Delta)=\ell, \ \card(\Delta\cap \mathcal{T}_n)\geq k,\ |J_n\cap\Delta|=0 \},
\end{align*}
where in both definitions we count the points in $\Delta\cap \mathcal T_n$ including multiplicities.
Then,
\begin{equation}\label{eq:Jnsum}
\frac{1}{|\Delta|} \sum_{n\in N(\Delta)} |J_n|\lesssim_{k} 1\qquad\text{and}\qquad \sum_{n\in M(\Delta,\ell)} \frac{|J_n|}{\dist(J_n,\Delta)+|\Delta|}\lesssim_{k,\gamma} (\ell+1)^2.
\end{equation}
\end{lem}
\begin{proof}
For every $n\in N(\Delta)$, there are only the $k-1$ possibilities $D_{m,i}^{(1)},$ $\dots,$  $D_{m,i+k-2}^{(1)}$ for $J_n$ and by Lemma \ref{lem:jinterval}, each interval $D_{m,j}^{(1)}$, $j=i,\dots,i+k-2$ occurs at most $F_k$ times as a characteristic interval. This implies the first inequality in \eqref{eq:Jnsum}.

We now prove the second inequality in \eqref{eq:Jnsum}. To begin with, assume that each $J_n$, $n\in M(\Delta,\ell)$ lies to the right of $\Delta$, since the other case is covered by similar methods. The argument is split in two parts depending on the value of the parameter $\ell$, beginning with $\ell\leq k$. In that case, for $n\in M(\Delta,\ell)$, let $J_n^{1/2}$ be the unique interval determined by the conditions
\[
\sup J_n^{1/2} = \sup J_n,\qquad |J_n^{1/2}|=|J_n|/2.
\]
Since $d_n(\Delta)=\ell$ is constant, we group the occurring intervals $J_n$ into packets, where all intervals in one packet have the same left endpoint and maximal intervals from different packets are disjoint (up to possibly one point). By Lemma \ref{lem:jinterval}, each point $t\in [0,1]$ belongs to at most $F_k$ intervals $J_n^{1/2}$. The $(k-1)$-regularity and the fact that $\ell\leq k$ now imply $|J_n|\lesssim_{k,\gamma} |\Delta|$ and $\dist(\Delta,J_n)\lesssim_{k,\gamma} |\Delta|$ for $n\in M(\Delta,\ell)$ and thus, every interval $J_n$ for $n\in M(\Delta,\ell)$ is a subset of a {fixed} interval whose length is comparable to $|\Delta|$.  So, putting these things together,
\begin{align*}
\sum_{n\in M(\Delta,\ell)} \frac{|J_n|}{\dist(J_n,\Delta)+|\Delta|}
&\leq \frac{1}{|\Delta|}\sum_{n\in M(\Delta,\ell)} |J_n| = \frac{2}{|\Delta|}\sum_{n\in M(\Delta,\ell)}\int_{J_n^{1/2}} \dif x \lesssim_{k,\gamma} 1,
\end{align*}
which completes the case of $\ell\leq k$.

Next, assume $\ell\geq k+1$ and define $(L_j)_{j=1}^\infty$ as the strictly decreasing sequence of all sets $L$ that satisfy
\[
L= D_{n,i}^{(k-1)}\qquad\text{and}\qquad \sup L=\sup \Delta
\]
for some index $n$ and $i$. Moreover, set
\[
M_j(\Delta,\ell):=\{n\in M(\Delta,\ell): \card(L_j\cap\mathcal{T}_n{)}=k\},
\]
i.e., $L_j$ is a union of $k-1$ grid point intervals in the grid $\mathcal{T}_n$.
Then, since $|\Delta|+\dist(J_n,\Delta)\gtrsim_\gamma |\Delta|+\dist(t,\Delta)$ for $t\in J_n^{1/2}$ by $(k-1)$-regularity,
\begin{align*}
\sum_{n\in M_j(\Delta,\ell)} \frac{|J_n|}{\dist(J_n,\Delta)+|\Delta|} &\lesssim_{k,\gamma} \sum_{n\in M_j(\Delta,\ell)}\int_{J_n^{1/2}}\frac{1}{\dist(t,\Delta)+|\Delta|}\dif t.
\end{align*}
If $n\in M_j(\Delta,\ell)$ we get, again due to $(k-1)$-regularity,
\begin{align*}
\inf J_n^{1/2}\geq \inf J_n\geq \gamma^{-k} |L_j|+\sup\Delta,
\end{align*}
and
\[
\sup J_n^{1/2}\leq \inf J_n+|J_n|\leq C_k\gamma^\ell |L_j|+\sup\Delta
\]
for some constant $C_k$ only depending on $k$.
Combining this with Lemma \ref{lem:jinterval}, which implies that each point $t$ belongs to at most $F_k$ intervals $J_n^{1/2}$,
\begin{equation}\label{eq:intLj}
\sum_{n\in M_j(\Delta,\ell)}\int_{J_n^{1/2}}\frac{1}{\dist(t,\Delta)+|\Delta|}\dif t\lesssim \int_{\gamma^{-k}|L_j|+|\Delta|}^{C_k\gamma^\ell|L_j|+|\Delta|}\frac{1}{s}\dif s.
\end{equation}
Next we will show that the above intervals of integration can intersect at most for roughly $\ell$ indices $j$. Let $j_2\geq j_1$, so that $L_{j_1}\supset L_{j_2}$ and write $j_2=j_1+2kr+t$ with $t\leq 2k-1$. Then, by Lemma \ref{lem:geom},
\[
C_k\gamma^\ell |L_{j_2}|\leq C_k\gamma^\ell |L_{j_1+2kr}|\leq C_k\gamma^\ell \eta^r |L_{j_1}|,
\]
where $\eta=\gamma^{k-1}/(1+\gamma^{k-1})<1$. If now $r\geq C_{k,\gamma}\ell$ for some suitable constant $C_{k,\gamma}$ depending only on $k$ and $\gamma$, we have
\[
C_k\gamma^\ell |L_{j_2}|\leq \gamma^{-k}|L_{j_1}|.
\]
Thus, each point $s$ in the integral in \eqref{eq:intLj} for some $j$ belongs to at most $C_{k,\gamma}\ell$ intervals $[\gamma^{-k}|L_j|+|\Delta|,C_k\gamma^\ell|L_j|+|\Delta|]$ where $j$ is varying. So we conclude by summing over $j$
\[
\sum_{n\in M(\Delta,\ell)} \frac{|J_n|}{\dist(J_n,\Delta)+|\Delta|} \leq C_{k,\gamma}\ell \int_{|\Delta|}^{(1+C_k\gamma^\ell)|\Delta|}\frac{1}{s}\dif s \leq C_{k,\gamma} \ell^2.
\]
{This completes the analysis of the case $\ell\geq k+1$ and thus, the proof of the lemma is finished.}
\end{proof}
\section{{Four  conditions on spline series and their relations}}\label{four.cond}
Let $(t_n)$ be a $k$-admissible sequence of knots with the corresponding orthonormal spline system $(f_n)_{n\geq -k+2}$. For a sequence $(a_n)_{n\geq -k+2}$ of coefficients, let
\[
P:=\Big(\sum_{n=-k+2}^\infty a_n^2 f_n^2\Big)^{1/2}\qquad\text{and}\qquad S:=\max_{m\geq -k+2}\Big|\sum_{n=-k+2}^m a_nf_n\Big|.
\]
If $f\in L^1$, we denote by $Pf$ and $Sf$ the functions $P$ and $S$ corresponding to the coefficient sequence $a_n=\langle f,f_n\rangle$, respectively.
Consider the following conditions:
\begin{enumerate}[(A)]
\item $P\in L^1$ \label{en:A},
\item The series $\sum_{n=-k+2}^\infty a_n f_n$ converges unconditionally in $L^1$, \label{en:B}
\item $S\in L^1$, \label{en:C}
\item There exists a function $f\in H^1$ such that $a_n=\langle f,f_n\rangle$. \label{en:D}
\end{enumerate}
We will discuss the relations between those four conditions and we will prove the implications indicated in the subsequent picture, where some results need certain regularity conditions imposed on the point sequence $(t_n)$, which is also indicated in the image.
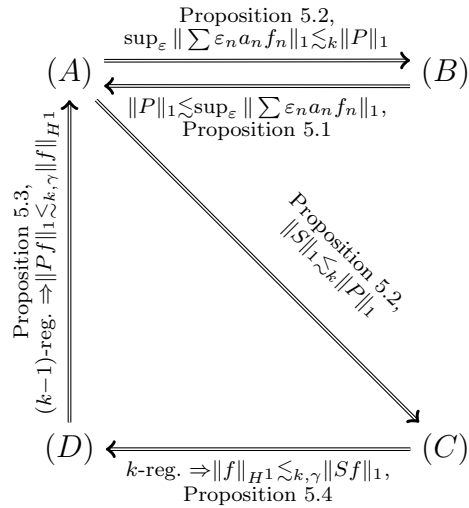
\begin{center}
\begin{tikzpicture}\label{fig:equivalences}
  \node (A) {$(A)$};
  \node (B) [node distance=5cm, right of=A] {$(B)$};
  \node (C) [node distance=5cm, below of=B] {$(C)$};
  \node (D) [node distance=5cm, below of=A] {$(D)$};
  \draw[->, double] (A.20) to node [above] { $\substack{\text{Proposition \ref{prop:A->B_A->C}}, \\ \sup_{\varepsilon}\|\sum\varepsilon_na_nf_n\|_1\lesssim_k \|P\|_1}$} (B.160);
  \draw[<-, double] (A.340) to node [below] {$\substack{\|P\|_1\lesssim \sup_{\varepsilon}\|\sum\varepsilon_na_nf_n\|_1,\\ \text{Proposition \ref{prop:B->A}}}$} (B.200);
  \draw[->, double] (A) to node [right] {\rotatebox{-45}{$\substack{\text{Proposition \ref{prop:A->B_A->C}}, \\ \|S\|_1\lesssim_k \|P\|_1}$}} (C);
  \draw[->, double] (C) to node [below] {$\substack{k\text{-reg. }\Rightarrow \|f\|_{H^1}\lesssim_{k,\gamma}\|Sf\|_1,\\ \text{Proposition \ref{prop:C->D}}}$} (D);
  \draw[->, double] (D) to node [left]
  {\rotatebox{90}{$\substack{\text{Proposition \ref{prop:D->A}},\\(k-1)\text{-reg. }\Rightarrow \|Pf\|_1\lesssim_{k,\gamma} \|f\|_{H^1}}$}} (A);
\end{tikzpicture}
\end{center}

Let us recall that in case of orthonormal spline systems with dyadic knots, the relations (and equivalences) of these conditions have been studied by several
authors, also in case $p<1$, see e.g. \cite{SjolinStromberg1983a,ChangCiesielski1983,Gevorkyan1989}. For general Franklin systems corresponding to arbitrary sequences of knots, the relations of these conditions were discussed in \cite{GevorkyanKamont2005} (and earlier in \cite{GevorkyanKamont1998}, also for $p<1$, but for a restricted class of {point sequences}). In the sequel,   we follow the approach from \cite{GevorkyanKamont2005} and we adapt it to the case of spline orthonormal systems of order $k$.

\bigskip

We begin with the implication $\eqref{en:B}\Rightarrow\eqref{en:A}$, which is a consequence of Khinchin's inequality:
\begin{prop}[$\eqref{en:B}\Rightarrow\eqref{en:A}$]\label{prop:B->A}
Let $(t_n)$ be a $k$-admissible sequence of knots with the corresponding general orthonormal spline system $(f_n)$ and let $(a_n)$ be a sequence of coefficients. If the series $\sum_{n=-k+2}^\infty a_nf_n$ converges unconditionally in $L^1$, then $P\in L^1$. Moreover,
\[
\|P\|_1\lesssim \sup_{\varepsilon\in \{-1,1\}^{\mathbb{Z}}}\big\|\sum_{n=-k+2}^\infty\varepsilon_n a_n f_n\big\|_1.
\]
\end{prop}

\bigskip

Next, we investigate the implications $\eqref{en:A}\Rightarrow\eqref{en:B}$ and $\eqref{en:A}\Rightarrow\eqref{en:C}$. Let us note that once we know the estimates and combinatorial  results of Sections \ref{sec:prel} and \ref{sec:proporth},
the proof is the same as the proof of Proposition 4.3 in \cite{GevorkyanKamont2005}, so we just state the result.

\begin{prop}[$\eqref{en:A}\Rightarrow\eqref{en:B}$ and $\eqref{en:A}\Rightarrow\eqref{en:C}$]\label{prop:A->B_A->C}
Let $(t_n)$ be a $k$-admissible sequence of knots and let $(a_n)$ be a sequence of coefficients such that $P\in L^1$. Then, $S\in L^1$ and $\sum a_n f_n$ converges unconditionally in $L^1$; moreover,
\[
\sup_{\varepsilon\in\{-1,1\}^{\mathbb{Z}}}\big\|\sum_{n=-k+2}^\infty\varepsilon_na_nf_n\big\|\lesssim_k \|P\|_1\qquad\text{and}\qquad
\|S\|_1\lesssim_k \|P\|_1 .
\]
\end{prop}

\bigskip

Next we discuss $\eqref{en:D}\Rightarrow\eqref{en:A}$.

\begin{prop}[$\eqref{en:D}\Rightarrow\eqref{en:A}$]\label{prop:D->A}
Let $(t_n)$ be a $k$-admissible point sequence that satisfies the $(k-1)$-regularity condition with parameter $\gamma$. Then there exists a constant $C_{k,\gamma}$ depending only on $k$ and $\gamma$ such that for each atom $\phi$,
\[
\|P\phi\|_1\leq C_{k,\gamma}.
\]
Consequently, if $f \in H^1$, then
$$
\| Pf \|_1 \leq C_{k,\gamma} \| f \|_{H^1}.
$$
\end{prop}

Before we proceed with the proof, let us remark that essentially the same arguments give a direct proof of $\eqref{en:D}\Rightarrow\eqref{en:C}$,
under the same assumption of $(k-1)$-regularity of the sequence of points $(t_n)$, and moreover
$$
\| Sf \|_1 \leq C_{k,\gamma} \| f \|_{H^1}.
$$
We do not present it here, since we have the  implications $\eqref{en:D}\Rightarrow\eqref{en:A}$ under the assumption of $(k-1)$-regularity and
$\eqref{en:A}\Rightarrow\eqref{en:C}$ under the assumption of $k$-admissibility only.
Note that Proposition \ref{prop:not_D->A} in Section \ref{main.proof} shows that -- without the assumption of $(k-1)$-regularity of the point sequence --
the implications  $\eqref{en:D}\Rightarrow\eqref{en:A}$ and $\eqref{en:D}\Rightarrow\eqref{en:C}$ need not be true.

\begin{proof}[Proof of Proposition \ref{prop:D->A}]
Let $\phi$ be an atom with $\int_0^1\phi(u)\dif u=0$ and let $\Gamma=[\alpha,\beta]$ be an interval such that $\supp\phi \subset \Gamma$ and $\sup|\phi|\leq |\Gamma|^{-1}$. Define $n_\Gamma:=\max\{n:\card(\mathcal{T}_n\cap \Gamma)\leq k-1\}$, where in the maximum, we also count multiplicities of knots.
It will be shown that
\[
\|P_1\phi\|_1,\ \|P_2\phi\|_1 \lesssim_{\gamma,k} 1,
\]
where
\[
P_1\phi=\Big(\sum_{n\leq n_\Gamma} a_n^2 f_n^2\Big)^{1/2}\qquad\text{and}\qquad P_2\phi=\Big(\sum_{n> n_\Gamma} a_n^2 f_n^2\Big)^{1/2}.
\]
First, we consider $P_1$ and prove the stronger inequality
\[
\sum_{n\leq n_\Gamma} |a_n|\|f_n\|_1\lesssim_{k,\gamma}1,
\]
where $a_n=\langle\phi,f_n\rangle$. For each parameter $n\leq n_\Gamma$, we define $\Gamma_{n,\alpha}$ as the unique closed interval $D_{n,j}^{(k-1)}$ with minimal index $j$ such that
\[
\alpha\leq \min D_{n,j+1}^{(k-1)}.
\]
We note that $\Gamma_{n,\alpha}$ satisfies
\[
\Gamma_{n_1,\alpha} \supseteq \Gamma_{n_2,\alpha}\qquad\text{for } n_1\leq n_2,
\]
and, by $(k-1)$-regularity,
\[
|\Gamma_{n,\alpha}|\gtrsim_{\gamma,k} |\Gamma|.
\]
Let $g_n=\sum_{j=1}^{n+k-1} w_jN_{n,j}^{(k)}$ be the unnormalized orthogonal spline function as in \eqref{eq:defg} and the coefficients $(w_j)$ as in \eqref{eq:defwj}. For $\xi\in \Gamma$, we have (cf. \eqref{eq:splinederivative})
\begin{align}\label{eq:sumprime}
|g_n'(\xi)|\lesssim_k \sum_j \frac{|w_j|+|w_{j-1}|}{|D_{n,j}^{(k-1)}|},
\end{align}
where we sum only over those indices $j$ such that $\Gamma\cap\supp N_{n,j}^{(k-1)}=\Gamma\cap D_{n,j}^{(k-1)}\neq\emptyset$. By $(k-1)$-regularity, all lengths $|D_{n,j}^{(k-1)}|$ in this summation range are comparable to $|\Gamma_{n,\alpha}|$. Moreover, by \eqref{eq:wj},
\[
|w_j|\lesssim_k \frac{q^{d_n(\tau_{n,j})}}{|J_n|+\dist(D_{n,j}^{(k)},J_n)+|D_{n,j}^{(k)}|}.
\]
Again by $(k-1)$-regularity, for $j$ in the summation range of the sum \eqref{eq:sumprime},
\begin{align*}
|D_{n,j}^{(k-1)}|&\gtrsim_{k,\gamma} |\Gamma_{n,\alpha}|, \\
\dist(D_{n,j}^{(k)}, J_n)+|D_{n,j}^{(k)}|&\gtrsim_{k,\gamma} \dist(J_n,\Gamma_{n,\alpha})+|\Gamma_{n,\alpha}|.
\end{align*}
Therefore, combining the above inequalities, we estimate the right hand side in \eqref{eq:sumprime} further and get, with the notation $\Gamma_n:=\Gamma_{n,\alpha}$,
\begin{equation}\label{eq:primebound}
|g_n'(\xi)|\lesssim_{k,\gamma} \frac{1}{|\Gamma_{n}|}\frac{q^{d_n(\Gamma_{n})}}{|J_n|+\dist(J_n,\Gamma_{n})+|\Gamma_{n}|}.
\end{equation}
As a consequence, for an arbitrary point $\tau\in\Gamma$,
\begin{align*}
|a_n|&=\Big|\int_{\Gamma} \phi(t)[f_n(t)-f_n(\tau)]\dif t \Big| \leq \int_{\Gamma}\frac{1}{|\Gamma|} \sup_{\xi\in\Gamma} |f_n'(\xi)||t-\tau|\dif t \\
&\lesssim_k |\Gamma||J_n|^{1/2}\sup_{\xi\in\Gamma}|g_n'(\xi)|\lesssim_{k,\gamma} \frac{|\Gamma|}{|\Gamma_{n}|} \frac{|J_n|^{1/2}q^{d_n(\Gamma_{n})}}{|J_n|+\dist(J_n,\Gamma_{n})+|\Gamma_{n}|}.
\end{align*}
Let $\Delta_1\supset \cdots\supset \Delta_s$ be the collection of all different intervals appearing as $\Gamma_n$ for $n\leq n_\Gamma$. By Lemma \ref{lem:geom}, we have some geometric decay in the measure of $\Delta_i$. Now fix $\Delta_i$ and $\ell\geq 0$ and consider indices $n\leq n_\Gamma$ such that $\Gamma_n=\Delta_i$ and $d_n(\Gamma_n)=\ell$. By the latter display and Lemma \ref{lem:orthsplineJinterval},
\[
|a_n|\|f_n\|_1 \lesssim_{k,\gamma} \frac{|\Gamma|}{|\Delta_i|} \frac{|J_n|q^{\ell}}{|J_n|+\dist(J_n,\Delta_i)+|\Delta_i|},
\]
and thus, Lemma \ref{lem:sumJn} implies
\[
\sum_{n:\Gamma_{n}=\Delta_i,\, d_n(\Gamma_n)=\ell} |a_n|\|f_n\|_1\lesssim_{k,\gamma} (\ell+1)^2 q^\ell \frac{|\Gamma|}{|\Delta_i|}.
\]
Now, summing over $\ell$ and then over $i$ (recall that $|\Delta_i|$ decays like a geometric progression by Lemma \ref{lem:geom} and $|\Delta_i|\gtrsim_{k,\gamma} |\Gamma|$ since $n\leq n_\Gamma$) yields
\[
\sum_{n\leq n_\Gamma} |a_n|\|f_n\|_1 \lesssim_{k,\gamma} 1.
\]
This implies the desired inequality $\|P_1\phi\|_1 \lesssim_{k,\gamma} 1$ for the first part of $P\phi$.

Next, we look at $P_2\phi$ and define the set $V$ as the smallest interval that has grid points in $\mathcal T_{n_\Gamma +1}$ as endpoints and which contains $\Gamma$. Moreover, $\widetilde{V}$ is defined to be the smallest interval with gridpoints in $\mathcal T_{n_\Gamma +1}$ as endpoints and such that $\widetilde{V}$ contains $k$ grid points in $\mathcal T_{n_\Gamma +1}$ to the left of $\Gamma$ and as well $k$ grid points in $\mathcal T_{n_\Gamma +1}$ to the right of $\Gamma$.
We observe that due to $(k-1)$-regularity and the fact that $\Gamma$ contains at least $k$ gridpoints from $\mathcal T_{n_\Gamma+1}$,
\begin{equation}\label{eq:simVGamma}
\begin{aligned}
|V|&\sim_{k,\gamma} |\widetilde{V}|\sim_{k,\gamma} |\Gamma|,\\
 |(\widetilde{V}\setminus V)\cap[0,\inf\Gamma]|&\sim_{k,\gamma} |(\widetilde{V}\setminus V)\cap[\sup\Gamma,1]|\sim_{k,\gamma} |\widetilde{V}|.
\end{aligned}
\end{equation}

First, we consider the integral of $P_2\phi$ over the set $\widetilde{V}$ and obtain by the Cauchy-Schwarz inequality
\[
\int_{\widetilde{V}} P_2\phi(t)\dif t\leq \|\charfun_{\widetilde{V}}\|_2 \|\phi\|_2\leq \frac{|\widetilde{V}|^{1/2}}{|\Gamma|^{1/2}} \lesssim_{k,\gamma} 1.
\]

It remains to estimate $\int_{\widetilde{V}^c} P_2\phi(t)\dif t $.
Since for $n>n_\Gamma$, the endpoints of $\widetilde{V}$ are grid points in $\mathcal{T}_n$, for $J_n$ there are only the possibilities $J_n\subset\widetilde{V}$, $J_n$ is to the right or $J_n$ is to the left of $\widetilde{V}$. We begin with considering $J_n\subset\widetilde{V}$, in which case
\[
|a_n|=\Big|\int_{\Gamma} \phi(t)f_n(t)\dif t\Big| \leq \frac{\|f_n\|_1}{|\Gamma|}\lesssim_k \frac{|J_n|^{1/2}}{|\Gamma|},
\]
and therefore, by Lemma \ref{lem:JnsubsetV} and \eqref{eq:simVGamma},
\begin{align*}
\sum_{n:J_n\subset\widetilde{V},n>n_\Gamma} |a_n|\int_{\widetilde{V}^c}|f_n(t)|\dif t &\lesssim_k \frac{1}{|\Gamma|} \sum_{n:J_n\subset\widetilde{V}} |J_n|^{1/2} \int_{\widetilde{V}^c}|f_n(t)|\dif t  \\
&\lesssim_{k} \frac{|\widetilde{V}|}{|\Gamma|}\lesssim_{k,\gamma} 1.
\end{align*}

Now, let $J_n$ be to the right of $\widetilde{V}$. The case of $J_n$ to the left of $\widetilde{V}$ is then considered similarly.
By \eqref{eq:phiplinks} for $p=\infty$,
\[
|a_n|\leq \frac{1}{|\Gamma|}\int_{\Gamma}|f_n(t)|\dif t  \leq  \frac{1}{|\Gamma|}\int_{V}|f_n(t)|\dif t
\lesssim_{k,\gamma} \frac{q^{d_n(V)}|J_n|^{1/2}}{\dist(V,J_n)+|J_n|}.
\]
This inequality, Lemma \ref{lem:orthsplineJinterval} and the fact that $\dist(V,J_n)\gtrsim_{k,\gamma} \dist(V,J_n)+|V|$ (cf. \eqref{eq:simVGamma}) allow us to deduce
\begin{align*}
\sum_{\substack{n>n_\Gamma \\ J_n\text{ is to the right of }\widetilde{V}}} |a_n|\|f_n\|_1\lesssim_{k,\gamma} \sum_{\substack{n>n_\Gamma \\ J_n\text{ is to the right of }\widetilde{V}}} \frac{q^{d_n(V)}|J_n|}{\dist(V,J_n)+|V|}.
\end{align*}
Note that $V$ can be a union of $k-1$, $k$ or $k+1$ intervals from $\mathcal T_{n_\Gamma +1}$; therefore, let  $V^+$ be a union of
$k-1$ grid intervals form $\mathcal T_{n_\Gamma +1}$, with right endpoint of $V^+$ coinciding with the right endpoint of $V$.
Then, if $J_n$ is to the right of $V$ then $d_n(V) = d_n(V^+)$, $\dist(V,J_n) = \dist(V^+,J_n)$ and --
because of $(k-1)$-regularity of the {point sequence} -- $|V|
\sim_{k,\gamma} |V^+|$, which implies
$$
 \sum_{\substack{n>n_\Gamma \\ J_n\text{ is to the right of }\widetilde{V}}} \frac{q^{d_n(V)}|J_n|}{\dist(V,J_n)+|V|}
 \lesssim_{k,\gamma}
  \sum_{\substack{n>n_\Gamma \\ J_n\text{ is to the right of }\widetilde{V}}} \frac{q^{d_n(V^+)}|J_n|}{\dist(V^+,J_n)+|V^+|}.
$$

Finally, we employ Lemma \ref{lem:sumJn} to conclude
\begin{align*}
\sum_{\substack{n>n_\Gamma \\ J_n\text{ is to the right of }\widetilde{V}}} |a_n|\|f_n\|_1
&\lesssim_{k,\gamma}\sum_{\ell=0}^{\infty}q^\ell \sum_{\substack{n>n_\Gamma \\ d_n(V^+)=\ell \\ J_n\text{ is to the right of }\widetilde{V}}} \frac{|J_n|}{\dist(V^+,J_n)+|V^+|} \\
&\lesssim_{k,\gamma} \sum_{\ell=0}^\infty (\ell+1)^2q^\ell \lesssim_{k} 1.
\end{align*}

To conclude the proof, note that if $f \in H^1$ and $f = \sum_{m=1}^\infty c_m \phi_m$ is an atomic decomposition
of $f$, then $\langle f, f_n\rangle = \sum_{m=1}^\infty c_m  \langle \phi_m,f_n\rangle$, and
$ Pf (t) \leq \sum_{m=1}^\infty |c_m| P\phi_m (t)$.
\end{proof}

\bigskip

Finally, we discuss $\eqref{en:C}\Rightarrow\eqref{en:D}$.

\begin{prop}[$\eqref{en:C}\Rightarrow\eqref{en:D}$]\label{prop:C->D}
Let $(t_n)$ be a $k$-admissible sequence of knots in $[0,1]$ satisfying the $k$-regularity condition with parameter $\gamma$ and let $(a_n)$ be a sequence of coefficients such that $S=\sup_{m}\big|\sum_{n\leq m} a_n f_n\big|\in L^1$. Then, there exists a function $f\in H^1$ with $a_n=\langle f,f_n\rangle$ for each $n$. Moreover, we have the inequality
\[
\|f\|_{H^1}\lesssim_{k,\gamma} \|Sf\|_1.
\]
\end{prop}

\begin{proof}
If $S\in L^1$, then there is a function $f\in L^1$ such that $f=\sum_{n\geq -k+2} a_n f_n$ with convergence in $L^1$. Indeed, this is a consequence of the relative weak compactness of uniformly integrable subsets in $L^1$ and the basis property of $(f_n)$ in $L^1$. Thus, we need only show that $f\in H^1$ and this is done by finding a suitable atomic decomposition of $f$.

We define $E_0=B_0=[0,1]$ and, for $r\geq 1$,
\begin{align*}
E_r=[S>2^r],\qquad B_r=[\mathcal M\charfun_{E_r}> c_{k,\gamma}],
\end{align*}
where $\mathcal M$ denotes the Hardy-Littlewood maximal function and $0<c_{k,\gamma}\leq 1/2$ is a small constant only depending on $k$ and $\gamma$ which is chosen according to a few restrictions that will be given during the proof. We note that
\[
\mathcal M \charfun_{E_r}(t)=\sup_{I\ni t}\frac{|I\cap E_r|}{|I|},\qquad t\in[0,1],
\]
where the supremum is taken over all intervals containing the point $t$.
 Since the Hardy-Littlewood maximal function {operator} $\mathcal M$ is of weak type $(1,1)$, we have the inequality $|B_r|\lesssim_{k,\gamma} |E_r|$. Due to the fact that $S\in L^1$, $|E_r|\to 0$ and, as a consequence, $|B_r|\to 0$ as $r\to\infty$. Now, decompose the open set $B_r$ into a countable union of disjoint open intervals
\[
B_r=\bigcup_{\kappa} \Gamma_{r,\kappa},
\]
{where for fixed $r$, no two intervals $\Gamma_{r, \kappa}$ have a common endpoint and the above equality is up to measure $0$ {(each open set of real numbers can be decomposed into a countable union of open intervals, but it can happen that two intervals have the same endpoint. In that case, we collect those two intervals to one $\Gamma_{r,\kappa}$).
This can be achieved by taking as $\Gamma_{r,\kappa}$ the collection of level sets of positive measure of
the function $t \to | [0,t] \cap B_r^c|$.}}

Moreover, observe that if $\Gamma_{r+1,\xi}$ is one of the intervals in the decomposition of $B_{r+1}$, then there is an interval $\Gamma_{r,\kappa}$ in the decomposition of $B_r$ such that $\Gamma_{r+1,\xi}\subset \Gamma_{r,\kappa}$.

Based on this decomposition, we define the following functions for $r\geq 0$:
\begin{equation*}
g_r(t):=\begin{cases}
f(t), & t\in B_r^c, \\
\frac{1}{|\Gamma_{r,\kappa}|}\int_{\Gamma_{r,\kappa}} f(t)\dif t, & t\in\Gamma_{r,\kappa}.
\end{cases}
\end{equation*}
Observe that $f=g_0+\sum_{r=0}^\infty (g_{r+1}-g_r)$ in $L^1$ 
and $g_{r+1}-g_r=0$ on $B_r^c$. As a consequence,
\begin{align*}
\int_{\Gamma_{r,\kappa}} g_{r+1}(t)\dif t
&=\int_{\Gamma_{r,\kappa}\cap B_{r+1}^c} g_{r+1}(t)\dif t + \int_{\Gamma_{r,\kappa}\cap B_{r+1}} g_{r+1}(t)\dif t \\
&=\int_{\Gamma_{r,\kappa}\cap B_{r+1}^c} f(t)\dif t + \sum_{\xi\,:\, \Gamma_{r+1,\xi}\subset \Gamma_{r,\kappa}} \int_{\Gamma_{r+1,\xi}} f(t)\dif t \\
&= \int_{\Gamma_{r,\kappa}} f(t)\dif t=  \int_{\Gamma_{r,\kappa}} g_r(t)\dif t.
\end{align*}
The main step of the proof is to show that
\begin{equation}\label{eq:gpointwise}
|g_r(t)|\leq C_{k,\gamma} 2^r,\qquad \text{a.e. } t\in [0,1]
\end{equation}
for some constant $C_{k,\gamma}$ only depending on $k$ and $\gamma$.
Once this inequality is proved, we take $\phi_0\equiv 1$, $\eta_0=\int_0^1 f(u)\dif u$ and
\[
\phi_{r,\kappa}:=\frac{(g_{r+1}-g_r)\charfun_{\Gamma_{r,\kappa}}}{C_{k,\gamma}2^r|\Gamma_{r,\kappa}|},\qquad \eta_{r,\kappa}=C_{k,\gamma} 2^r |\Gamma_{r,\kappa}|
\]
and observe that $f=\eta_0\phi_0+\sum_{r,\kappa}\eta_{r,\kappa}\phi_{r,\kappa}$ is the desired atomic decomposition of $f$ since
\begin{align*}
\sum_{r,\kappa}\eta_{r,\kappa}
&\leq C_{k,\gamma} \sum_{r,\kappa} 2^r |\Gamma_{r,\kappa}| =C_{k,\gamma}\sum_{r} 2^r |B_{r}|  \\
&\lesssim_{k,\gamma} \sum_{r} 2^r |E_r| \lesssim \|S\|_1.
\end{align*}
Thus it remains to prove inequality \eqref{eq:gpointwise}.

In order to do this, we first assume $t\in B_r^c$. Additionally, assume that $t$ is a point such that the series $\sum_n a_nf_n(t)$ converges to $f(t)$ and that $t$ does not occur in the grid point sequence $(t_n)$. By Theorem \ref{thm:ae}, this holds for a.e. point in $[0,1]$. We fix the index $m$ and let $V_m$ be the maximal interval where the function $S_m:=\sum_{n\leq m} a_nf_n$ is a polynomial of order $k$ that contains the point $t$. Then, $V_m\not\subset B_r$ and since $V_m$ is an interval containing the point $t$,
\[
|V_m\cap E_r^c|\geq (1-c_{k,\gamma}) |V_m|\geq |V_m|/2.
\]
Since $|S_m|\leq 2^r$ on $E_r^c$, the above display and Proposition \ref{prop:poly} imply $|S_m|\lesssim_{k} 2^r$ on $V_m$ and in particular $|S_m(t)|\lesssim_k 2^r$. Now, $S_m(t)\to f(t)$ as $m\to\infty$ by the assumptions on $t$ and thus,
\[
|g_r(t)|=|f(t)|\lesssim_k 2^r.
\]
This concludes the proof of \eqref{eq:gpointwise} in the case of $t\in B_r^c$.

Next, we fix an index $\kappa$ and consider $g_r$ on $\Gamma:=[\alpha,\beta]:=\Gamma_{r,\kappa}$. Let $n_\Gamma$ be the first index such that there are $k+1$ points from $\mathcal T_{n_\Gamma}$ contained in $\Gamma$, i.e., there exists a support $D_{n_\Gamma,i}^{(k)}$ of a B-spline function of order $k$ in the grid $\mathcal T_{n_\Gamma}$ that is contained in $\Gamma$. Additionally, we define the intervals
\[
U_0:=[\tau_{n_\Gamma,i-k},\tau_{n_\Gamma,i}],\qquad W_0:=[\tau_{n_\Gamma,i+k},\tau_{n_\Gamma,i+2k}].
\]
Note that if  $\alpha \in \mathcal T_{n_\Gamma}$, then $\alpha$  is a common endpoint of
$U_0$ and $\Gamma$, otherwise $\alpha$ is an interior point of $U_0$. Similarly, if   $\beta \in \mathcal T_{n_\Gamma}$, then $\beta$  is a common endpoint of
$W_0$ and $\Gamma$, otherwise $\beta$ is an interior point of $W_0$.
By $k$-regularity of the point sequence $(t_n)$, $\max(|U_0|,|W_0|)\lesssim_{k,\gamma}  |\Gamma|$.
We first estimate the part $S_\Gamma:=\sum_{n\leq n_\Gamma} a_n f_n$ and show the inequality $|S_\Gamma|\lesssim_{k,\gamma} 2^r$ on $\Gamma$. Observe that on $\Delta:=U_0\cup\Gamma\cup W_0$, $S_\Gamma$ can be represented as a linear combination of B-splines $(N_j)$ on the grid $\mathcal T_{n_\Gamma}$ of the form
\[
S_\Gamma(t)=h(t):=\sum_{j=i-2k+1}^{i+2k-1} b_j N_j(t),
\]
for some coefficients $(b_j)$. For $j=i-2k+1,\dots,i+2k-1$, let $J_j$ be a maximal interval of $\supp N_j$ and observe that due to $k$-regularity, $|J_j|\sim_{k,\gamma} |\Gamma|\sim_{k,\gamma}|\supp h|$.

 If we assume that $\max_{J_j}|S_\Gamma|> C_{k} 2^r$, where $C_k$ is the constant obtained from Proposition \ref{prop:poly} with $\rho=1/2$, then Proposition \ref{prop:poly} implies that $|S_\Gamma|> 2^r$ on a subset $I_j$ of $J_j$ with measure $\geq |J_j|/2$. As a consequence, 
\[
|\supp h\cap E_r|\geq |J_j\cap E_r|\geq  |J_j|/2\gtrsim_{k,\gamma} |\supp h|.
\]
We choose the constant $c_{k,\gamma}$ in the definition of $B_r$ sufficiently small to guarantee that this last inequality implies $\supp h\subset B_r$. This is a contradiction to the choice of $\Gamma$, which implies that our assumption $\max_{J_j}|S_\Gamma|>C_k2^r$ is not true and thus
\[
\max_{J_j}|S_{\Gamma}|\leq C_k2^r,\qquad j=i-2k+1,\dots,i+2k-1.
\]
By local stability of B-splines, i.e., inequality \eqref{eq:lpstab} in Proposition \ref{prop:lpstab},
this implies
\[
|b_j|\lesssim_k 2^r,\qquad j=i-2k+1,\dots,i+2k-1,
\]
and so $|S_\Gamma|\lesssim_k 2^r$ on $\Delta$. This means
\begin{equation}\label{eq:firstpart}
\int_{\Gamma} |S_\Gamma| \lesssim_k 2^r|\Gamma|,
\end{equation}
which is inequality \eqref{eq:gpointwise} for the part $S_\Gamma$.

In order to estimate the remaining part, we inductively define two sequences $(u_s,U_s)_{i\geq 0}$ and $(w_s,W_s)_{s\geq 0}$ consisting of integers and intervals. Set $u_0=w_0=n_\Gamma$ and inductively define $u_{s+1}$ to be the first index $n>u_s$ such that $t_n\in U_s$. Moreover, define $U_{s+1}$ to be the B-spline support $D_{u_{s+1},i^{(k)}}$ in the grid $\mathcal T_{u_{s+1}}$, where $i$ is the minimal index such that $D_{u_{s+1},i}^{(k)}\cap\Gamma\neq\emptyset$.
 Similarly, we define $w_{s+1}$ to be the first index $n>w_s$ such that $t_n\in W_s$ and $W_{s+1}$ as the B-spline support $D_{w_{s+1},i}^{(k)}$ in the grid $\mathcal T_{w_{s+1}}$, where $i$ is the maximal index such that $D_{w_{s+1},i}^{(k)}\cap \Gamma\neq\emptyset$. It can easily be seen that this construction implies $U_{s+1}\subset U_s$, $W_{s+1}\subset W_s$ and $\alpha\in U_s$, $\beta\in W_s$ for all $s\geq 0$, or more precisely:
 if $\alpha \in \mathcal T_{u_{s}}$, then $\alpha$ is a common endpoint of $U_{s}$ and $\Gamma$, otherwise $\alpha$ is an inner point of $U_{s}$,
 and similarly, if $\beta \in \mathcal T_{u_{s}}$, then $\beta$ is a common endpoint of $W_{s}$ and $\Gamma$, otherwise $\beta$ is an inner point of $W_{s}$.

For a pair of indices $\ell,m$, let
\[
x_\ell:=\sum_{\nu=0}^{k-1}N_{u_\ell,i+\nu}\charfun_{U_\ell},\qquad y_m:=\sum_{\nu=0}^{k-1}N_{w_m,j-\nu}\charfun_{W_m},
\]
where $N_{u_\ell,i}$ is the B-spline function on the grid $\mathcal T_{u_\ell}$ with support $U_\ell$ and $N_{w_m,j}$ is the B-spline function on the grid $\mathcal T_{w_m}$ with support $W_m$. The function
\[
\phi_{\ell,m}:=x_\ell + \charfun_{\Gamma\setminus (U_\ell\cup W_m)}+y_m
\]
is zero on $(U_\ell\cup \Gamma\cup W_m)^c$, one on $\Gamma\setminus(U_\ell\cup W_m)$ and a piecewise polynomial function of order $k$ in between. For $\ell,m\geq 0$, consider the following subsets of $\{n: n>n_\Gamma\}$:
\[
L(\ell):=\{n:u_\ell<n\leq u_{\ell+1}\},\qquad R(m):=\{n:w_m<n\leq w_{m+1}\}.
\]
If $n\in L(\ell)\cap R(m)$, we clearly have $\langle f_n,\phi_{\ell,m}\rangle=0$ and thus
\begin{equation}
\int_{\Gamma} f_n(t)\dif t=\int_{\Gamma} f_n(t)\dif t-\int_0^1 f_n(t)\phi_{\ell,m}(t)\dif t=A_\ell(f_n)+B_m(f_n),
\end{equation}
where
\begin{align*}
A_\ell(f_n)&:=\int_{\Gamma\cap U_\ell} f_n(t)\dif t-\int_{U_\ell} f_n(t)x_\ell(t)\dif t, \\ B_m(f_n)&:=\int_{\Gamma\cap W_m} f_n(t)\dif t-\int_{W_m} f_n(t)y_m(t)\dif t.
\end{align*}
This implies
\begin{equation}
\label{eq:splittingUW}
\begin{aligned}
\Big|\int_\Gamma \sum_{n=n_\Gamma+1}^\infty &a_nf_n(t)\dif t\Big|
= \Big|\sum_{\ell,m=0}^{\infty}\sum_{n\in L(\ell)\cap R(m)}a_n\big(A_\ell(f_n)+B_m(f_n)\big)\Big| \\
&\leq 2\sum_{\ell=0}^{\infty}\int_{U_\ell}\Big|\sum_{n\in L(\ell)}a_nf_n(t)\Big|\dif t +  2\sum_{m=0}^{\infty}\int_{W_m}\Big|\sum_{n\in R(m)}a_nf_n(t)\Big|\dif t.
\end{aligned}
\end{equation}
Consider the first sum on the right hand side. On $U_\ell=D_{u_\ell,i}^{(k)}$, the function $\sum_{n\in L(\ell)} a_n f_n$ can be represented as a linear combination of B-splines $(N_j)$ on the grid $\mathcal T_{u_\ell}$ of the form
\[
\sum_{n\in L(\ell)} a_nf_n=h_\ell:=\sum_{j=i-k+1}^{i+k-1} b_j N_j,
\]
for some coefficients $(b_j)$. For $j=i-k+1,\dots, i+k-1$, let $J_j$ be a maximal grid interval of $\supp N_j$ and observe that due to $k$-regularity, $|J_j|\sim_{k,\gamma} |U_\ell|\sim_{k,\gamma} |\supp h_\ell|$.
 On $J_j$, the function $\sum_{n\in L(\ell)} a_nf_n$ is a polynomial of order $k$. If we assume $\max_{J_j} \big|\sum_{n\in L(\ell)} a_nf_n\big|>C_k2^{r+1}$, where $C_k$ is the constant obtained from Proposition \ref{prop:poly} with $\rho=1/2$, then Proposition \ref{prop:poly} implies that $\big|\sum_{n\in L(\ell)} a_nf_n\big|>2^{r+1}$ on a set $J_j^*\subset J_j$ with $|J_j^*|=|J_j|/2$, but this means $\max\big(|\sum_{n\leq u_\ell} a_nf_n\big|,|\sum_{n\leq u_{\ell+1}} a_nf_n\big|\big)>2^r$ on the set $J^*_j$. As a consequence,
\[
| E_r \cap \supp h_\ell|\geq|E_r\cap J_j|\geq |J_j^*|\geq |J_j|/2\gtrsim_{k}|\supp h_\ell|.
\]
We choose the constant $c_{k,\gamma}$ in the definition of $B_r$ sufficiently small to guarantee that this last inequality implies $\supp h_\ell\subset B_r$. This is a contradiction to the choice of $\Gamma$, which implies that our assumption $\max_{J_j}\big|\sum_{n\in L(\ell)} a_nf_n\big|>C_k2^r$ is not true and thus
\[
\max_{J_j}\big|\sum_{n\in L(\ell)} a_nf_n\big|\leq C_k2^r,\qquad j=i-k+1,\dots,i+k-1.
\]
By local stability of B-splines, i.e., inequality \eqref{eq:lpstab} in Proposition \ref{prop:lpstab},
this implies
\[
|b_j|\lesssim_k 2^r,\qquad j=i-k+1,\dots,i+k-1,
\]
and so $\big|\sum_{n\in L(\ell)} a_nf_n\big|\lesssim_k 2^r$ on $U_\ell$, which gives
\[
\int_{U_\ell} \big|\sum_{n\in L(\ell)} a_nf_n\big|\lesssim_k 2^r|U_\ell|.
\]
Combining Lemma \ref{lem:geom}, the inclusions $U_{\ell+1}\subset U_\ell$ and the inequality $|U_0|\lesssim_{k,\gamma} |\Gamma|$, we see that $\sum_{\ell=0}^{\infty}|U_\ell|\lesssim_{k,\gamma} |\Gamma|$. Thus we get
\[
\sum_{\ell=0}^{\infty}\int_{U_\ell} \big|\sum_{n\in L(\ell)} a_nf_n\big|\lesssim_{k,\gamma} 2^r|\Gamma|.
\]
The second sum on the right hand side of \eqref{eq:splittingUW} is estimated similarly which gives
\[
\sum_{m=0}^{\infty}\int_{W_m} \big|\sum_{n\in R(m)} a_nf_n\big|\lesssim_{k,\gamma} 2^r|\Gamma|.
\]
Combining these estimates with \eqref{eq:splittingUW} and \eqref{eq:firstpart}, we find
\[
\Big|\int_{\Gamma}f(t)\dif t\Big|=\Big|\int_{\Gamma}\sum_n a_nf_n(t)\dif t\Big|\lesssim_{k,\gamma}2^r|\Gamma|,
\]
which implies \eqref{eq:gpointwise} on $\Gamma$ and therefore, the proof of the propositon is completed.
\end{proof}

\section{Proof of the main theorem}\label{main.proof}

For the proof of the necessity part of Theorem \ref{thm:uncond}, we will use the following:

\begin{prop}\label{prop:not_D->A}
Let $(t_n)$ be a $k$-admissible sequence of knots satisfying the $k$-regularity condition with parameter $\gamma$, but which is not satisfying any $(k-1)$-regularity condition. Then,
\[
\sup \| \sup_n |a_n(\phi)f_n|\|_1=\infty,
\]
where $\sup$ is taken over all atoms $\phi$ and $a_n(\phi):=\langle\phi,f_n\rangle$.
\end{prop}

 Proposition \ref{prop:not_D->A} implies in particular  that Proposition \ref{prop:D->A} cannot be extended to arbitrary partitions.
 For the proof of Proposition \ref{prop:not_D->A} we need the following technical Lemma \ref{lem:comb1}.

\begin{lem}\label{lem:comb1}
Let $(t_n)$ be a $k$-admissible sequence of knots satisfying the $k$-regularity condition with parameter $\gamma\geq 1$, but is not satisfying any $(k-1)$-regularity condition. Let $\ell$ be an arbitrary positive integer. Then, for all $A\geq 2$, there exists a finite increasing sequence $(n_j)_{j=0}^{\ell-1}$ of length $\ell$ such that if $\tau_{n_j,i_j}$ is the new point in $\mathcal T_{n_j}$ that is not present in $\mathcal T_{n_j-1}$ and
\begin{align*}
\Lambda_j:=[\tau_{n_j,i_j-k},\tau_{n_j,i_j-1}),\qquad L_j:=[\tau_{n_j,i_j-1},\tau_{n_j,i_j}),\qquad R_j:=[\tau_{n_j,i_j},\tau_{n_j,i_j+1}),
\end{align*}
we have for all indices $i,j$ in the range $0\leq i<j\leq \ell-1$
\begin{enumerate}
\item $R_i\cap R_j=\emptyset$, \label{it:1}
\item $\Lambda_i=\Lambda_j$, \label{it:2}
\item $(2 \gamma -1) |L_j| \geq |[\tau_{n_j,i_j-k-1},\tau_{n_j,i_j-k}]|\geq  {\frac{| L_j|}{2 \gamma}}$, \label{it:3}
\item \label{it:3.5}
$|R_j|\leq (2\gamma-1)|L_j|$,
\item \label{it:4}
$|L_j|\leq 2(\gamma+1)k\cdot|R_j|$,
\item $\min(|L_j|,|R_j|)\geq  A|\Lambda_j|$. \label{it:6}
\end{enumerate}
\end{lem}

\begin{proof} First, we choose a sequence $(n_j)_{j=0}^{lk}$ so that \eqref{it:1} -- \eqref{it:3.5} hold.
Next, we choose $(n_{m_j})_{j=0}^{l-1}$ -- a subsequence of $(n_j)_{j=0}^{lk}$ -- so that  \eqref{it:4} and \eqref{it:6} hold as well.

The sequence $(t_n)$ does not satisfy any $(k-1)$-regularity condition. As a consequence, for all $C_0$ there exists an index $n_0$ and an index $i_0$ such that
\begin{equation}\label{eq:regC0}
{\rm either} \quad C_0 |D^{(k-1)}_{n_0,i_0-k}|  \leq  |D^{(k-1)}_{n_0,i_0-k+1}| \quad
{\rm or} \quad  |D^{(k-1)}_{n_0,i_0-k}|  \geq  C_0 |D^{(k-1)}_{n_0,i_0-k+1}|.
\end{equation}
We choose $C_0$ sufficiently large such that {with} $C_j:=C_{j-1}/\gamma-1$ for $j\geq 1$, the condition $C_{k\ell}\geq 2\gamma$ is true. We will give an additional restriction for $C_0$ at the end of the proof.
Without loss of generality, we can assume that the first inequality in \eqref{eq:regC0} holds. Taking $\Lambda_0 = [\tau_{n_0,i_0 -k}, \tau_{n_0,i_0-1})$
and $L_0 = [\tau_{n_0,i_0 -1}, \tau_{n_0,i_0})$, $R_0 = [\tau_{n_0,i_0}, \tau_{n_0,i_0+1})$,
we have
\begin{equation}\label{eq:C0}
|[\tau_{n_0,i_0-k+1},\tau_{n_0,i_0}]|\geq C_0|\Lambda_0|.
\end{equation}
A direct consequence of \eqref{eq:C0} is
\begin{equation}\label{eq:C0-1}
|L_0| \geq (C_0-1) |\Lambda_0|.
\end{equation}
By $k$-regularity we have
$$
| D^{(k)} _{n_0,i_0-k-1}| \geq  { |D^{(k)}_{n_0,i_0-k}| \over \gamma} = { |\Lambda_0| + |L_0| \over \gamma} ,
$$
which implies
\begin{equation} \label{eq:ind-0}
\begin{aligned}
|[\tau_{n_0,i_0-k-1},\tau_{n_0,i_0-k}]| & = |D_{n_0,i_0-k-1}^{(k)}|-|\Lambda_0| \\
&\geq \frac{|\Lambda_0|+|L_0|}{\gamma}-|\Lambda_0|\\
& \geq \frac{|L_0|}{2\gamma}+\frac{|\Lambda_0|}{\gamma}+\frac{C_0-1}{2\gamma}|\Lambda_0|-|\Lambda_0| \\
& = {|L_0| \over 2 \gamma} + \Big({C_0 +1 \over 2 \gamma} - 1\Big) |\Lambda_0| \geq { |L_0| \over 2 \gamma},
\end{aligned}
\end{equation}
i.e., the right hand side inequality of \eqref{it:3} for $j=0$.
To get the upper estimate, note that by $k$-regularity
$$
|\Lambda_0| + |[\tau_{n_0,i_0-k-1},\tau_{n_0,i_0-k}]| \leq \gamma (|\Lambda_0| + |L_0|) ,
$$
hence by \eqref{eq:C0-1}
\begin{equation}\label{eq:ind-1}
|[\tau_{n_0,i_0-k-1},\tau_{n_0,i_0-k}]| \leq \gamma |L_0| + (\gamma - 1) |\Lambda_0| \leq (2 \gamma -1) |L_0|.
\end{equation}
This and the previous calculation give (3) for $j=0$.
Therefore, the construction below can be continued either to the right or to the left of $\Lambda_0$.
We continue our construction to the right of $\Lambda_0$.

We continue by induction. Having defined $n_j$, $\Lambda_j$, $L_j$ and $R_j$, we take
\[
n_{j+1}:=\min\{n>n_{j}:t_n\in \Lambda_j\cup L_j\},\qquad j\geq 0.
\]
By definition of $R_j$ and $n_{j+1}$, property \eqref{it:1} is satisfied for all $j\geq 0$.
We identify $t_{n_{j+1}} = \tau_{n_{j+1},i_{j+1}}$.
Thus, by construction,   $t_{n_j} = \tau_{n_j,i_j}$ is a common endpoint of $L_j$ and $R_j$ for $j \geq 1$.

In order to prove \eqref{it:2}, we will show by induction that
\begin{equation}\label{eq:ind}
|[\tau_{n_j,i_j-k+1},\tau_{n_j,i_j}]|\geq C_j|\Lambda_j|
\qquad\text{and}\qquad \Lambda_{j+1}=\Lambda_j
\end{equation}
for all $j=0,\dots,k\ell$.
We remark that the equation $\Lambda_{j+1}=\Lambda_j$ is equivalent to the condition $\tau_{n_{j+1},i_{j+1}}\in L_j$.

The first inequality of \eqref{eq:ind} for $j=0$ is exactly \eqref{eq:C0}.
If the second identity in \eqref{eq:ind} were not satisfied for $j=0$, i.e.,
$\tau_{n_1,i_1}\in \Lambda_0$, we would have by $k$-regularity of the point sequence $(t_n)$, applied to the partition $\mathcal T_{n_1}$,
\[
|\Lambda_0| \geq \frac{1}{\gamma}|L_0|,
\]
which contradicts \eqref{eq:C0-1} for our choice of $C_0$. This means we have $\Lambda_1=\Lambda_0$, and so, \eqref{eq:ind} is true for $j=0$. Next, assume that \eqref{eq:ind} is satisfied for $j-1$, where $j$ is in the range $1\leq j\leq k\ell-1$.
By $k$-regularity, applied to the partition $\mathcal T_{n_j}$, and employing repeatedly  \eqref{eq:ind} for $j-1$,
\begin{align*}
|\Lambda_j|+|L_j|=|\Lambda_j\cup L_j|
&\geq \frac{1}{\gamma}(\tau_{n_j,i_j+1}-\tau_{n_j,i_j-k+1})\\
&= \frac{1}{\gamma}(\tau_{n_{j-1},i_{j-1}}-\tau_{n_{j-1},i_{j-1}-k+1}) \\
&\geq \frac{C_{j-1}}{\gamma}|\Lambda_{j-1}|=\frac{C_{j-1}}{\gamma}|\Lambda_j|.
\end{align*}
This means, by the recursive definition of $C_j$,
\begin{equation}\label{eq:Cj}
|L_j|\geq C_j|\Lambda_j|,
\end{equation}
and in particular the first identity in \eqref{eq:ind} is true for $j$. If the second identity in \eqref{eq:ind} were not satisfied for $j$, i.e., $\tau_{n_{j+1},i_{j+1}}\in \Lambda_j$, we would have by $k$-regularity of the point sequence $(t_n)$, applied to the partition $\mathcal T_{n_{j+1}}$,
\[
|\Lambda_j|\geq \frac{1}{\gamma} |L_j|,
\]
which contradicts \eqref{eq:Cj} and our choice of $C_0$. This proves \eqref{eq:ind} for $j$ and thus, property \eqref{it:2} is true for all $j=0,\dots,k\ell$.

Moreover, choosing $C_0$ sufficiently large -- that is, such that $C_{kl} \geq 2(\gamma+1) k A$, \eqref{eq:Cj} implies
\begin{equation}\label{eq:RL-1}
|L_j| \geq 2(\gamma+1) k A |\Lambda_j|,
\end{equation}
which is a part of \eqref{it:6}.

The lower estimate in property \eqref{it:3} is proved by repeating the argument giving \eqref{eq:ind-0} and using \eqref{eq:Cj} instead of \eqref{eq:ind-0}. Moreover, the upper estimate also uses the same arguments as the proof of \eqref{eq:ind-1}, but now we have to use  \eqref{eq:Cj} as well.

Next, we look at property \eqref{it:3.5}. By $k$-regularity and \eqref{eq:Cj}, as $C_j >1$, we have
\[
|R_j|+|L_j|\leq \gamma (|L_j|+|\Lambda_j|)\leq 2\gamma |L_j|,
\]
which is exactly property \eqref{it:3.5}.

We prove property \eqref{it:4} by choosing a suitable subsequence of $(n_j)_{j=0}^{k\ell}$ and begin with assuming the contrary to \eqref{it:4} for $k$ consecutive indices, i.e., for an index $s$,
\begin{equation}\label{eq:RL}
|R_{s+r}|< \alpha|L_{s+r}|\leq \alpha|L_s|,\qquad r=1,\dots,k,
\end{equation}
where $\alpha:=\big(2(\gamma+1)k\big)^{-1}$.  We have $L_j=L_{j+1}\cup R_{j+1}$ for $0\leq j\leq k\ell-1$. Thus, on the one hand,
\begin{equation}\label{eq:contr1}
|L_s\setminus L_{s+k}|=\sum_{r=1}^k |R_{s+r}|\leq \alpha k |L_s|
\end{equation}
by \eqref{eq:RL}; on the other hand, by $k$-regularity of the partition $\mathcal{T}_{n_{s+k}}$,
\begin{equation}\label{eq:contr2}
\begin{aligned}
|L_s\setminus L_{s+k}|\geq \frac{1}{\gamma}|L_{s+k}|&=\frac{1}{\gamma}\Big(|L_s|-\sum_{r=1}^k |R_{s+r}|\Big) \geq \frac{1-\alpha k}{\gamma}|L_s|.
\end{aligned}
\end{equation}
Now, \eqref{eq:contr1} contradicts \eqref{eq:contr2} for our choice of $\alpha$.
We thus have proved that for $k$ consecutive indices $s+1,\dots,s+k$, there is at least one index $s+r$, $1\leq r\leq k$,
such that \eqref{it:4} is satisfied for $s+r$.
As a consequence we can extract a sequence of length $\ell$ from $(n_j)_{j=1}^{k\ell}$ satisfying \eqref{it:4}.
Without restriction, this subsequence is called $(n_j)_{j=0}^{\ell-1}$ again.

Property \eqref{it:6} for $R_j$
is now a simple consequence of \eqref{eq:RL-1}, property \eqref{it:4} and the choice of the subsequence $(n_j)_{j=0}^{\ell-1}$.
 Therefore, the proof of the lemma is completed.
\end{proof}


Now, we are ready to proceed with the proof of  Proposition \ref{prop:not_D->A}.

\begin{proof}[Proof of Proposition \ref{prop:not_D->A}. ]
Let $\ell$ be an arbitrary positive integer and $A\geq 2$ a number that will be chosen later. Then, Lemma \ref{lem:comb1} gives us a sequence $(n_j)_{j=0}^{\ell-1}$ such that all conditions in Lemma \ref{lem:comb1} are satisfied. We assume that $|\Lambda_0|>0$. Let $\tau:=\tau_{n_0,i_0-1}$, $x:=\tau-2|\Lambda_0|$ and $y:=\tau+2|\Lambda_0|$. Then we define the atom $\phi$ by
\[
\phi\equiv\frac{1}{4|\Lambda_0|}(\charfun_{[x,\tau]}-\charfun_{[\tau,y]})
\]
and let $j$ be an arbitrary integer in the range $0\leq j\leq \ell-1$. By partial integration, the expression $a_{n_j}(\phi)=\langle \phi,f_{n_j}\rangle$ can be written as
\begin{equation*}
\begin{aligned}
4|\Lambda_0| a_{n_j}(\phi)
&=\int_x^\tau f_{n_j}(t)\dif t-\int_{\tau}^{y} f_{n_j}(t)\dif t \\
&=\int_x^\tau f_{n_j}(t)-f_{n_j}(\tau)\dif t-\int_{\tau}^{y} f_{n_j}(t)-f_{n_j}(\tau)\dif t \\
&=\int_{x}^\tau (x-t)f_{n_j}'(t)\dif t- \int_\tau^y (y-t)f_{n_j}'(t)\dif t.
\end{aligned}
\end{equation*}
In order to estimate $|a_{n_j}(\phi)|$ from below, we estimate the absolute values of $I_1:=\int_{x}^\tau (x-t)f_{n_j}'(t)\dif t$ from below and of $I_2:=\int_\tau^y (y-t)f_{n_j}'(t)\dif t$ from above and  begin with $I_2$.

Consider the function $g_{n_j}$, which is connected to
 $f_{n_j}$ via $f_{n_j}=g_{n_j}/\|g_{n_j}\|_2$ and $\|g_{n_j}\|_2\sim_k |J_{n_j}|^{-1/2}$ (cf. \eqref{eq:defg} and Lemma \ref{lem:orthsplineJinterval}).
 In the notation of Lemma \ref{lem:comb1}, $g_{n_j}$ is obtained by inserting the point $t_{n_j} = \tau_{n_j,i_j}$ to ${\mathcal T}_{n_j -1}$, and
 it is a common endpoint of intervals $L_i$ and $R_i$.
 By construction of the characteristic interval $J_{n_j}$,
Lemma \ref{lem:comb1} properties \eqref{it:3.5} -- \eqref{it:6}, and the $k$-regularity of the point sequence $(t_n)$,
we have
\begin{equation}\label{eq:equiv}
|J_{n_j}|\sim_{k,\gamma} |L_j|\sim_{k,\gamma} |R_j|.
\end{equation}

By Lemma \ref{lem:comb1}, property \eqref{it:6}, we have $[\tau,y]\subset L_j$, and therefore on $[\tau,y]$,
the derivative of the function $g_{n_j}$ has the representation (cf.
\eqref{eq:splinederivative})
\[
g_{n_j}'(u)=(k-1)\sum_{i=i_j-k+1}^{i_j-1}\xi_i N_{n_j,i}^{(k-1)}(u),\qquad u\in[\tau,y],
\]
where $\xi_i=(w_i-w_{i-1})/|D_{n_j,i}^{(k-1)}|$ and the coefficients $w_i$ are given by $\eqref{eq:defwj}$ associated to the partition $\mathcal T_{n_j}$.
For $i=i_j-k+1,\dots i_j-1$ we have $L_j \subset D_{n_j,i}^{(k-1)}$, which in combination with the $k$-regularity of the point sequence $(t_n)$ and
 Lemma \ref{lem:comb1} property \eqref{it:6} implies
\begin{equation}\label{eq:equiv-1}
|J_{n_j}| \sim_k |L_j|\sim_{k,\gamma}  |D_{n_j,i}^{(k-1)}|,\qquad i=i_j-k+1,\dots i_j-1.
\end{equation}
Moreover, by Lemma \ref{lem:lporthspline},
\[
|w_i|\lesssim_k \frac{1}{|J_{n_j}|},\qquad 1\leq i\leq n_j+k-1.
\]
Therefore
$$
|f'_{n_j}(t) | \sim_{k} |J_{n_j}|^{1/2} |g'_{n_j}(t)| \lesssim_{k,\gamma} |L_j|^{-3/2} \quad {\rm for } \quad t \in [\tau,y].
$$
Consequently, putting the above facts together,
\begin{equation}\label{eq:II}
|I_2|\lesssim_{k,\gamma} |\Lambda_0|^2\cdot |L_j|^{-3/2}.
\end{equation}

We continue with the estimate of $I_1$. By properties \eqref{it:3} and \eqref{it:6}  of Lemma \ref{lem:comb1} (with $A \geq 2 \gamma$),
we have $[x,\tau]\subset [\tau_{n_j,i_j-k-1},\tau_{n_j,i_j-1}]$ and, therefore, on $[x,\tau]$, $g_{n_j}'$ has the representation (cf.
\eqref{eq:splinederivative})
\[
g_{n_j}'(u)=(k-1)\sum_{i=i_j-2k+1}^{i_j-2}\xi_i N_{n_j,i}^{(k-1)}(u),\qquad u\in[x,\tau].
\]
We split the integral $I_1$ as a sum $I_1 = I_{1,1} + I_{1,2}$ corresponding to the indices $i \neq i_{j} -k$ and $i = i_j -k$ in the above representation
of $g_{n_j}$ on $[x,\tau]$.

Note that $[\tau_{n_j,i_j-k-1},\tau_{n_j,i_j-k}] \subset D^{(k-1)}_{n_j, i}$ for $i_j-2k+1 \leq i< i_j - k$
and $L_j \subset D^{(k-1)}_{n_j, i}$ for $i_j -k < i \leq i_j-2$. Therefore, by properties \eqref{it:3} and \eqref{it:6} of Lemma \ref{lem:comb1}  and the $k$-regularity of the sequence of knots
we have
$$
| D^{(k-1)}_{n_j, i}| \sim_{k,\gamma} |L_j| \quad {\rm for} \quad i_j - 2k+ 1 \leq i \leq i_j -2, \quad i \neq i_j -k.
$$
{So}, by arguments analogous to the proof of \eqref{eq:II} we get
\begin{equation}\label{eq:Iupper}
| I_{1,1} | \sim_k |J_{n_j}|^{1/2} \Big|\int_x^\tau (t-x)\sum_{\substack{i=i_j-2k+1\\i\neq i_j-k}}^{i_j-2}\xi_i N_{n_j,i}^{(k-1)}(t)\dif t\Big|\lesssim_{k,\gamma} |\Lambda_0|^2\cdot |L_j|^{-3/2}.
\end{equation}
Moreover, for $i=i_j-k$, we have $D^{(k-1)}_{n_j,i_j-k} = \Lambda_0$, so we get
\begin{equation}\label{eq:Icont}
\begin{aligned}
| I_{1,2} | & \sim_k |J_{n_j}|^{1/2}\Big|\int_x^\tau (t-x)\xi_{i_j-k} N_{n_j,i_j-k}^{(k-1)}(t)\dif t\Big|
\\ & \geq |\xi_{i_j-k}| |J_{n_j}|^{1/2} |\Lambda_0| \int_x^\tau N_{n_j,i_j-k}^{(k-1)}(t)\dif t \\
&=|\xi_{i_j-k}||\Lambda_0| |J_{n_j}|^{1/2} \frac{|D_{n_j,i_j-k}^{(k-1)}|}{k-1}
 = |\xi_{i_j-k}| |J_{n_j}|^{1/2} \frac{|\Lambda_0|^2}{k-1},
\end{aligned}
\end{equation}
due to the fact that $t-x\geq |\Lambda_0|$ for $t\in\supp N_{n_j,i_j-k}^{(k-1)}$.
Since the sequence $w_j$ is checkerboard, cf. \eqref{eq:betragreinziehen},
\[
|\xi_{i_j-k}|=\frac{|w_{i_j-k}|+|w_{i_j-k-1}|}{|D_{n_j,i_j-k}^{(k-1)}|} \geq \frac{|w_{i_j-k}|}{|D_{n_j,i_j-k}^{(k-1)}|}.
\]
By definition of $w_{i_j-k}$,
\[
|w_{i_j-k}|\geq |\alpha_{i_j-k}||b_{i_j-k,i_j-k}|,
\]
where $\alpha_{i_j-k}$ is the factor from formula \eqref{eq:alpha2} and $b_{i_j-k,i_j-k}$ is an entry of the inverse of the B-spline Gram matrix, both corresponding to the partition $\mathcal T_{n_j}$. Formulas \eqref{eq:alpha2} and \eqref{eq:equiv} imply that $\alpha_{i_j-k}$ is bounded from below by a positive constant that is only depending on $k$ and $\gamma$.\footnote{Formula \eqref{eq:alpha2} is applied with $\mathcal T_n = \mathcal T_{n_j}$ and corresponding to
$\tau_{i_0} = \tau_{n_j,i_j}$. Then $[\tau_{i_0-1}, \tau_{i_0}] = L_j$ and $[\tau_{i_0}, \tau_{i_0+1}] = R_j$. Because of $k$-regularity and $|\Lambda_0 \cup L_j| \sim_{k,\gamma} |L_j|$, each denominator in \eqref{eq:alpha2} is $\sim_{k,\gamma} |L_j|$. Each nominator in \eqref{eq:alpha2}
 is bigger than either $L_j$ or $R_j$, so by \eqref{eq:equiv} and $k$-regularity it is $\sim_{k,\gamma} |L_j|$ as well.}
  Moreover, $|b_{i_j-k,i_j-k}|\geq \|N_{n_j,i_j-k}^{(k)}\|_2^{-2}\gtrsim_k |D_{n_j,i_j-k}^{(k)}|^{-1}$, cf. \eqref{estbi:lower}.
Note that $D^{(k)}_{n_j, i_j-k} = \Lambda_0 \cup L_j$, so $| D^{(k)}_{n_j, i_j-k} | \sim_{k,\gamma } |L_j|$.
Thus, $|\xi_{i_j-k}| \gtrsim_{k,\gamma} |\Lambda_0|^{-1} |L_j|^{-1}$.
 Inserting the above calculations in \eqref{eq:Icont}, we find
\begin{equation}\label{eq:Icont2}
\begin{aligned}
| I_{1,2} |
&\gtrsim_{k,\gamma} 
|J_{n_j}|^{1/2} {|\Lambda_0| \over |L_j|} \sim_{k,\gamma} | \Lambda_0|  |L_j|^{-1/2}.
\end{aligned}
\end{equation}
We now impose conditions on the constant $A \geq 2 \gamma$ from the beginning of the proof and property \eqref{it:6} in Lemma \ref{lem:comb1}.
It follows by \eqref{eq:Icont2}, \eqref{eq:Iupper} and \eqref{eq:II} that there are $C_{k,\gamma}>0$ and $c_{k,\gamma}>0$, depending only on
$k$ and $\gamma$ such that
\begin{eqnarray*}
4 |\Lambda_0| |a_{n_j}(\phi)| & \geq & |I_{1,2}| - |I_{1,1}| - |I_2|
 \geq   C_{k,\gamma} |\Lambda_0| |L_j|^{-1/2} - c_{k,\gamma} |\Lambda_0|^2 |L_j|^{-3/2}
\\ & = &
|\Lambda_0| |L_j|^{-1/2} (C_{k,\gamma} - c_{k,\gamma} |\Lambda_0| |L_j|^{-1}).
\end{eqnarray*}
By property \eqref{it:6} in Lemma \ref{lem:comb1} we have $|\Lambda_0 | |L_j|^{-1} \leq 1 / A$.
 Choosing $A$ sufficiently large to guarantee
 $$
 C_{k,\gamma} - \frac{c_{k,\gamma}}{A} \geq \frac{C_{k,\gamma}}{2},
 $$
we get a constant $m_{k,\gamma}$, depending only on $k$ and $\gamma$ such that
\begin{equation}\label{eq:anj}
m_{k,\gamma} |L_j|^{-1/2}\leq  |a_{n_j}(\phi)|,\qquad j=0,\dots,\ell-1.
\end{equation}

Next, we estimate $\int_{R_j}|g_{n_j}(t)|\dif t$ from below. First, employ Proposition \ref{prop:lpstab}, property \eqref{it:6} of Lemma \ref{lem:comb1}  and the $k$-regularity of the point sequence $(t_n)$ to get
\[
\int_{R_j}|g_{n_j}(t)|\dif t\gtrsim_{k,\gamma} |R_j||w_{i_j}|,
\]
where $w_{i_j}$ corresponds to the partition $\mathcal T_{n_j}$. By definition of $w_{i_j}$,
\[
\int_{R_j}|g_{n_j}(t)|\dif t\gtrsim_{k,\gamma} |R_j||\alpha_{i_j}||b_{i_j,i_j}|.
\]
By arguments similar as above, $|\alpha_{i_j}|$ is bounded from below by a constant only depending on $k$ and $\gamma$, and $|b_{i_j,i_j}|\gtrsim_k |D_{n_j,i_j}^{(k)}|^{-1}$. Since by $k$-regularity, $|R_j|\sim_{k,\gamma} |D_{n_j,i_j}^{(k)}|$, we finally get
\[
\int_{R_j}|g_{n_j}(t)|\dif t\gtrsim_{k,\gamma} 1,
\]
which means for $f_{n_j}$ that
\[
\int_{R_j}|f_{n_j}(t)|\dif t\gtrsim_{k,\gamma} |J_{n_j}|^{1/2}\gtrsim_{k,\gamma} |L_j|^{1/2}.
\]
Combining this last estimate with \eqref{eq:anj} and \eqref{it:1} of Lemma \ref{lem:comb1},
\[
\int_0^1\sup_n |a_n(\phi)f_n(t)|\dif t\geq \sum_{j=1}^\ell \int_{R_j}|a_{n_j}(\phi)f_{n_j}(t)|\dif t\gtrsim_{k,\gamma} \ell.
\]
This construction applies to every positive integer $\ell$, proving the assertion of the proposition for $|\Lambda_0|>0$.

The case $|\Lambda_0|=0$ proceeds similarly, with the difference that the atom $\phi$ is defined as centered at the point $\tau_{n_0,i_0-1}$ and the length of the support is sufficiently small, depending on $\ell$ and  $|L_0|$.
\end{proof}


With Proposition \ref{prop:not_D->A} and the results of Section \ref{four.cond}, the proof of Theorem  \ref{thm:uncond} follows now the same line of arguments as the proof of Theorem 2.2 in \cite{GevorkyanKamont2005}, but we present it here for the sake of the completeness.

 \begin{proof}[Proof of Theorem \ref{thm:uncond}]
We start by proving the unconditional basis property of $(f_n)=(f_n^{(k)})$ assuming the $(k-1)$-regularity of $(t_n)$. If $(t_n)$ is $(k-1)$-regular, it is not difficult to check that it is also $k$-regular. As a consequence, Theorem \ref{thm:basis} implies that $(f_n$) is a basis in $H^1$. Let $f\in H^1$ with $f=\sum a_n f_n$ and $\varepsilon\in \{-1,1\}^{\mathbb{Z}}$. We need to prove convergence of the series $\sum \varepsilon_na_nf_n$ in $H^1$.
Let $m_1\leq m_2$, then
\begin{align*}
\big\|\sum_{n=m_1}^{m_2}\varepsilon_na_nf_n\big\|_{H^1}
&\lesssim_{k,\gamma} \big\|S\big(\sum_{n=m_1}^{m_2}\varepsilon_na_nf_n\big)\big\|_1 \lesssim_k \big\|P\big(\sum_{n=m_1}^{m_2}\varepsilon_na_nf_n\big)\big\|_1 \\
&= \big\|P\big(\sum_{n=m_1}^{m_2}a_nf_n\big)\big\|_1 \lesssim_{k,\gamma} \big\|\sum_{n=m_1}^{m_2} a_nf_n\big\|_{H^1}
\end{align*}
where we used Proposition \ref{prop:C->D}, Proposition \ref{prop:A->B_A->C} and Proposition \ref{prop:D->A}, respectively (cf. also the picture on page \pageref{fig:equivalences}). So, since $\sum a_n f_n$ converges in $H^1$, so does $f_\varepsilon :=\sum \varepsilon_n a_nf_n$ and the same calculation as above shows
\[
\|f_\varepsilon\|_{H^1}\lesssim_{k,\gamma} \|f\|_{H^1}.
\]
This implies that $(f_n)$ is an unconditional basis in $H^1$.

We now prove the converse, i.e., that $(f_n)$ being an unconditional basis in $H^1$ implies the $(k-1)$-regularity condition. First, if $(t_n)$ does not satisfy the $k$-regularity condition, $(f_n)$ is not a basis in $H^1$ by Theorem \ref{thm:basis}. Thus, it remains to consider the case when $(t_n)$ satisfies the $k$-regularity condition, but not the $(k-1)$-regularity condition. By Theorem \ref{thm:basis} again, $(f_n)$ is then a basis in $H^1$. Suppose that $(f_n)$ is an unconditional basis in $H^1$. Then, for $f=\sum a_n f_n$ and $\varepsilon\in\{-1,1\}^\mathbb{Z}$, the function $f_\varepsilon:=\sum \varepsilon_na_nf_n$ is also in $H^1$. Since $\|\cdot\|_1\leq \|\cdot\|_{H^1}$, the series $\sum a_nf_n$ also converges unconditionally in $L^1$, and thus Proposition \ref{prop:B->A} (i.e., Khinchin's inequality) implies
\[
\|Pf\|_1\lesssim \sup_\varepsilon \|f_\varepsilon\|_1\leq \sup_\varepsilon\|f_\varepsilon\|_{H^1}\lesssim \|f\|_{H^1},
\]
which is impossible due to Proposition \ref{prop:not_D->A}, even for atoms. This concludes the proof of Theorem \ref{thm:uncond}.
\end{proof}
{As an immediate consequence of Theorem \ref{thm:uncond}, a fifth equivalent condition to (A)-(D) is the unconditional convergence of $\sum_n a_n f_n$ in $H^1$:}
\begin{cor}\label{corr:kreg}
Let $(t_n)$ be a $k$-admissible and $(k-1)$-regular sequence of points, with $(f_n)$ -- the corresponding orthonormal spline system of order $k$.
Let $(a_n)$ be a sequence of coefficients. Then conditions (A) -- (D) from Section \ref{four.cond} are equivalent.

Moreover, they are equivalent to the following

\smallskip

{\rm (E)} The series $\sum_n a_n f_n$ converges unconditionally in $H^1$.

\smallskip

In addition, for $f \in H^1$, $f = \sum_n a_n f_n$ we have
$$
\| f \|_{H^1} \sim \| Sf \|_1 \sim \| Pf \|_1 \sim \sup_{\varepsilon \in \{-1,1\}^{\mathbb Z}} \| \sum_n \varepsilon_n a_n f_n \|_1,
$$
with the implied constants depending only on $k$ and  the parameter of $(k-1)$-regularity of the sequence $(t_n)$.
\end{cor}

\subsection*{Acknowledgments}
A.~Kamont is supported by Polish MNiSW grant N N201 607840 {and}
M.~Passenbrunner is supported by the Austrian Science Fund, FWF project P 23987-N18.
Part of this work was done while A.~Kamont was visiting the Department of
Analysis, J. Kepler University Linz in November 2013. This
stay in Linz was supported by FWF project P 23987-N18.
G.~Gevorkyan and K.~Keryan were supported by SCS RA grant 13-1A006.
\nocite{GevorkyanKamont1998}\nocite{GevorkyanSahakian2000}\nocite{GevKam2004}\nocite{GevorkyanKamont2005}\nocite{GevorkyanKamont2008}
\bibliographystyle{plain}
\bibliography{uncondH1}
\end{document}